\documentclass[11pt]{amsart}
\ifdefined\pdftrailerid\pdftrailerid{}\fi

\usepackage[T1]{fontenc}
\usepackage{lmodern}
\usepackage[expansion=false]{microtype}
\usepackage{amsmath,amssymb,mathtools}
\usepackage{booktabs,array}
\usepackage[margin=1in]{geometry}
\usepackage{xcolor}
\definecolor{navblue}{RGB}{0,0,255}
\usepackage[colorlinks=true,linkcolor=red,linktoc=page,citecolor=navblue,urlcolor=navblue]{hyperref}
\hypersetup{
  pdftitle={Fröberg's Conjecture for Products of Linear Forms: Split Powers and Inherited Syzygies},
  pdfauthor={Qihang Wang},
  pdfsubject={Hilbert functions of ideals generated by products of linear forms},
  pdfkeywords={Fröberg conjecture, Hilbert function, products of linear forms, powers of linear forms, inverse systems, syzygies, Frobenius dilation, computer-assisted proof}
}

\newtheorem{maintheorem}{Theorem}

\newtheorem{theorem}{Theorem}[section]
\newtheorem{proposition}[theorem]{Proposition}
\newtheorem{lemma}[theorem]{Lemma}
\newtheorem{corollary}[theorem]{Corollary}
\theoremstyle{remark}
\newtheorem{remark}[theorem]{Remark}
\newtheorem{question}[theorem]{Question}

\newcommand{\FF}{\mathbb{F}}
\newcommand{\HS}{\operatorname{HS}}
\newcommand{\rank}{\operatorname{rank}}
\newcommand{\trunc}[1]{\left[#1\right]_{+}}

\title[Split powers and inherited syzygies]
{Fr\"oberg's Conjecture for Products of Linear Forms:\\ Split Powers and Inherited Syzygies}
\author{Qihang Wang}
\address{School of Mathematical Sciences, Peking University, Beijing 100871, China}
\email{2501110049@stu.pku.edu.cn}
\date{}


\begin{document}

\begin{abstract}
Fr\"oberg's conjecture predicts the Hilbert series of a polynomial ring
modulo general forms. We study explicit products of linear forms over
fields of characteristic zero. In four variables, six $d$th powers fail
to attain the predicted series for $d\ge4$. We fix seven rational
directions, take five $d$th powers, and split the sixth generator between
the two remaining directions with exponents $\lceil d/2\rceil$ and
$\lfloor d/2\rfloor$. The quotient has the predicted series for
$1\le d\le30$. For every $\varepsilon>0$ and all sufficiently large $d$,
its Hilbert function agrees with the prediction in every degree
$q\le(11/5-\varepsilon)d$ or $q\ge(73/32+\varepsilon)d$. In contrast,
for $d\ge4$ the predicted series cannot occur if each of the six forms
has a linear factor of multiplicity at least $(15d+91)/16$.
In three variables, relations among five powers $L_i^e$ give relations
among the products $L_i^eM_i$. For five general pairs of linear forms
and $e\ge36$, these relations and the pure-power Hilbert function give
a lower bound whose error is at most $55$ in every degree and is zero
through degree $\lfloor(7e+5)/4\rfloor-2$. The maximum excess over the
comparison proposed by Shapiro is $(10/3)e+O(\sqrt e)$.
\end{abstract}

\maketitle
\clearpage

\setcounter{tocdepth}{1}
\tableofcontents
\clearpage
\section{Introduction}

Fröberg's conjecture predicts the Hilbert series of a polynomial ring modulo
general forms. It is known in at most three variables and for at most
one more generator than variables. Explicit generators with the
predicted series are scarce: for $n+1$ forms of degree $d$ in $n$
variables one may take $x_1^d,\ldots,x_n^d,(x_1+\cdots+x_n)^d$, and
beyond this case few explicit families are known. Fröberg, Lundqvist, Oneto, and
Shapiro~\cite[Section~2]{FLOS2018} asked for such families, and in
particular for products of linear forms with the generic Hilbert series.

This paper studies how linear factors interact with the defect of
powers of linear forms. Two phenomena emerge. In four variables, six
$d$th powers of linear forms fail for every $d\ge4$
\cite[Proposition~7.3]{BoijLundqvist2026}, but keeping five powers and
\emph{splitting} the sixth evenly between two fixed directions gives
the predicted Hilbert function for $d\le30$, and for large $d$ outside
a narrow window of degrees. Keeping instead a linear factor of
multiplicity close to $d$ always forces excess. In three variables,
relations among powers $L_i^e$ are \emph{inherited} by the products
$L_i^eM_i$. For five general such products these inherited relations
determine the Hilbert function up to an absolute constant, and they
explain where the comparison proposed by Shapiro fails.

\subsection*{Conventions}
Let $k$ be a field of characteristic zero, let $S=k[x_1,\ldots,x_n]$ with
the standard grading, and fix positive degrees $d_1,\ldots,d_r$. A
property holds for \emph{general} forms if it holds on a nonempty
Zariski-open subset of $S_{d_1}\times\cdots\times S_{d_r}$. For factored
generators, generality refers instead to the parameter space of their
linear factors. The nonempty open sets below may depend on the exponent;
when a statement concerns every graded degree, the same open set works for
all those degrees at that exponent. For a power series
$A(t)=\sum_{j\ge0}a_jt^j$, let $[A(t)]_+$ keep $a_j$ whenever
$a_0,\ldots,a_j$ are all positive and replace every other coefficient
by $0$. Fröberg~\cite{Froberg1985} proved that the series
\begin{equation}\label{eq:froberg-general}
 \trunc{\frac{\prod_{i=1}^{r}(1-t^{d_i})}{(1-t)^n}}
\end{equation}
is a lexicographic lower bound for the Hilbert series of
$S/(F_1,\ldots,F_r)$, and conjectured equality for general forms.
We call~\eqref{eq:froberg-general} the \emph{predicted series}.

\subsection{Splitting one power in four variables}
In $k[x_0,x_1,x_2,x_3]$ put
\[
 L_i=x_0+i x_1+i^2x_2+i^3x_3\quad(1\le i\le6),
 \qquad D=x_0+x_3,
\]
and let $I(d)$ be the ideal generated by
\begin{equation}\label{eq:seven-intro}
 L_1^d,\ L_2^d,\ L_3^d,\ L_4^d,\ L_5^d,\
 L_6^{\lceil d/2\rceil}D^{\lfloor d/2\rfloor}.
\end{equation}
The seven directions are rational and independent of $d$.

\begin{maintheorem}\label{thm:seven-intro}
Let $S=k[x_0,x_1,x_2,x_3]$. The Hilbert series of $S/I(d)$
equals $[(1-t^d)^6/(1-t)^4]_+$ for $1\le d\le30$.
For every $\varepsilon>0$ and every sufficiently large integer $d$,
the Hilbert functions agree in every grade $q$ satisfying
\[
 q\le(11/5-\varepsilon)d\qquad\text{or}\qquad
 q\ge(73/32+\varepsilon)d .
\]
\end{maintheorem}

The predicted series vanishes from grade $\approx2.24070\,d$ on, so
the only grades not covered for large $d$ lie in a window of width
$(13/160+2\varepsilon)d$ around this transition. For example, when $d=30$
the last nonzero piece has dimension $340$ in grade $65$.

The proof combines exact ranks in two seed degrees with a dilation
argument. Below the first predicted zero, what must be controlled is
the first Koszul homology of the six generators; above it, vanishing
propagates by multiplication. Frobenius dilation, in the form developed
by Dinh and Ha~\cite{DinhHa2026}, transports the seed calculations in
degrees $30$ and $32$ to prime multiples $30p$ and $32p$. To reach
every large $d$ we insert the missing linear factors \emph{one row at
a time}. Adding a factor $v$ to one row preserves the vanishing of first
homology provided the other five rows have no relation on the plane
$v=0$ in the relevant degree. That condition reduces, again by
Frobenius, to the statement that on each of the seven planes the five
surviving quadrics $L_i^2$ and $L_6D$ are linearly independent. This
is where the choice of directions enters.

Since general forms specialize to~\eqref{eq:seven-intro}, upper
semicontinuity and Fröberg's lexicographic inequality show that six
general forms of degree $d\le30$ in four variables have the predicted
series, and that for large $d$ they have the predicted Hilbert function
in the same two ranges. For $d\le9$ this is also covered by the
equal-degree results recalled in Appendix~\ref{sec:fixed-degrees}. The
same seven directions give explicit witnesses throughout $10\le d\le30$.
By specialization the same conclusions hold for six general products
of any common partition type that refines the balanced split
(Corollary~\ref{cor:partition-lefschetz}). This gives positive cases of
Problem~F of~\cite{FLOS2018}, which prescribes the same non-pure
partition on every generator, and a specialization for their Problem~I,
which asks which products of powers act with maximal rank on a generic
quotient. Whether $I(d)$ has the predicted series for every $d$ is open.
Section~\ref{sec:residues} reduces the passage from $d$ to $d+1$ to two
rank conditions. After the contributions of the six planes have been
taken, the remaining map has a source of dimension only linear in $d$.

Two results explain why a construction of this kind must be built with
care. A common substitution of degree $m$ into a successful tuple,
followed by polynomial mixing of degree at most $m/4$, leaves an excess
of order $m^3$ at the first predicted zero
(Theorem~\ref{thm:substitution}); in~\eqref{eq:seven-intro} the
factors are instead inserted row by row. More strikingly, no
perturbation of six powers that keeps a factor of high multiplicity can
succeed.

\begin{maintheorem}\label{thm:intro-obstruction}
Let $d\ge4$ and put
$\lambda=\lceil 3d/16\rceil$. If $F_1,\ldots,F_6\in k[x_1,x_2,x_3,x_4]_d$
and each $F_i$ is divisible by $\ell_i^{5\lambda+1}$ for some linear form
$\ell_i$, then the quotient by $(F_1,\ldots,F_6)$ is nonzero in degree
$12\lambda$, whereas the predicted series
$[(1-t^d)^6/(1-t)^4]_+$ vanishes in that degree. In particular, this
applies whenever each $F_i$ is divisible by a power of a linear form with
exponent at least $(15d+91)/16$.
\end{maintheorem}

The hypothesis allows nonzero forms only when $5\lambda+1\le d$, for
instance for $d=16$ and for every $d\ge91$. The proof
(Theorem~\ref{thm:factor-obstruction}) is a single inverse-system
witness: the $\lambda$th power of the product of six quadric cones, each
singular at one of the six dual points and passing through the other
five. Applied to the homogeneous Rogers--Szegő polynomial, it shows that
no fixed-order Taylor truncation at $q=1$ completes five powers of
linear forms to the predicted series in large degrees
(Corollary~\ref{cor:fixed-jet}).

\subsection{Inherited syzygies in three variables}
Shapiro~\cite[Theorem~1.1 and Corollary~1.2]{Shapiro2026} showed that
the answer to Problem~F is negative: by a fat-point criterion, defects of
powers can persist after multiplication by arbitrary additional factors,
in an interval where the predicted series is zero. Here we detect excess
in degrees where the prediction is still positive. The source is a
transfer of relations: if $\sum_i L_i^eA_i=0$, then
\begin{equation}\label{eq:intro-transfer}
 \sum_i L_i^eM_i\Bigl(A_i\prod_{h\ne i}M_h\Bigr)=0 .
\end{equation}
Thus every relation among the powers yields a relation among the products,
five degrees later.

Let $S=k[x,y,z]$ and $I_e=(L_1^eM_1,\ldots,L_5^eM_5)$ for general linear
forms. Two lower bounds for $\dim_k(S/I_e)_q$ are immediate: the value
for five general forms of degree $e+1$, which is the predicted coefficient
by Anick's theorem~\cite{Anick1986}, and the value for the pure powers
$J_e=(L_1^e,\ldots,L_5^e)\supseteq I_e$. Shapiro~\cite[Question~4.1]{Shapiro2026}
asked whether the larger of the two is always the generic Hilbert
function. The transfer~\eqref{eq:intro-transfer} gives a third bound.
Let $\mathcal Z_q$ be the space of relations among $L_1^e,\ldots,L_5^e$ of
total degree $q$, put $N(q)=\dim_kS_q$ with negative components zero,
and define
\[
 B_e(q)=\max\left\{
 N(q)-5N(q-e-1)+\dim_k\mathcal Z_{q-5},\,
 \dim_k\bigl(S/(L_1^e,\ldots,L_5^e)\bigr)_q\right\}.
\]
Both terms are explicit, since the Hilbert function of five ternary powers
is classical; Proposition~\ref{prop:explicit-comparison} writes $B_e$ as a
piecewise polynomial in $q$ and $e$.

\begin{maintheorem}\label{thm:intro-uniform-hilbert}
Let $S=k[x,y,z]$. For every integer $e\ge36$ and five general pairs of
linear forms, set $I_e=(L_1^eM_1,\ldots,L_5^eM_5)$. Then
\[
 0\le\dim_k(S/I_e)_q-B_e(q)\le55
 \qquad\text{for every integer }q.
\]
Moreover $\dim_k(S/I_e)_q=B_e(q)$ for every $q\le\lfloor(7e+5)/4\rfloor-2$.
\end{maintheorem}

The equality statement is Proposition~\ref{prop:exact-lower-ternary}. The
reason the error does not grow is a cancellation. In every degree the
excess $\dim_k(S/I_e)_q-B_e(q)$ equals $\min\{\dim_kE_q,\dim_kC_q\}$,
where $E_q$ counts relations among the products modulo inherited ones and
$C_q=(J_e/I_e)_q$. After a specialization placing all ten coefficient
points on a conic, restriction to the five lines $M_i=0$ embeds $E_q$
into five spaces of polynomials of degree at most $2\lambda-2$, where
$\lambda=4q-7e+5$, and compatibility at the ten pairwise intersections
imposes forty independent conditions. Hence $\dim_kE_q\le10\lambda-45$
for $(7e+2)/4\le q\le2e-2$. In the same range
$\dim_kC_q=(100-10\lambda)+\dim_kE_q$, so $\lambda$ cancels and
$\dim_kC_q\le55$; below this range $E$ is controlled by multiplication by
a linear form. Above degree $2e-2$, Macaulay's growth bound
and Gotzmann's persistence theorem finish the argument. No computer
calculation is used.

Theorem~\ref{thm:intro-uniform-hilbert} measures Shapiro's proposed
comparison precisely: it is correct up to $55$ outside the window
$(5e+12)/3\le q\le(7e+5)/4$, and inside the window its defect is
$\min\{\binom{3q-5e-10}2,70e-40q+50\}$ up to the same constant, with maximum
$(10/3)e+O(\sqrt e)$ (Corollary~\ref{cor:shapiro-defect}). For large $e$
the maximum is attained where the Hilbert function is known exactly.
Whether $B_e$ is exact in every degree is equivalent to two restriction
ranks (Proposition~\ref{prop:two-critical-degrees}). The following family
isolates the phenomenon with nonzero additional factors that need not be
general.

\begin{maintheorem}\label{thm:intro-inheritance}
For every integer $u\ge12$, let $d=3u+1$ and $q=5u+4$.
For general $L_1,\ldots,L_5$ and any nonzero $M_1,\ldots,M_5$ in
$k[x,y,z]_1$, one has
\[
 \dim_k\bigl(k[x,y,z]/(L_1^{d-1}M_1,\ldots,L_5^{d-1}M_5)\bigr)_q
 \ge \frac{5u^2-35u-70}{2}+1.
\]
The predicted coefficient is $(5u^2-35u-70)/2$, and the
pure-power comparison is smaller by $10u-111$. For general pairs
and $u\ge19$, the displayed lower bound is the exact dimension.
\end{maintheorem}

For example, in degree $99$ the quotient by five general products
$L_i^{57}M_i$ has dimension $536$, while the general-form and pure-power
values are $535$ and $456$. The proof uses only~\eqref{eq:intro-transfer}
and the residual sequence of a conic (Theorem~\ref{thm:ternary-family}).
When the additional factors have common degree $s$, the same mechanism
gives an excess of exact order $se$ on a linearly growing interval,
uniformly for $1\le s\le\sigma_0e$ with fixed $\sigma_0<1/27$; for $s$
proportional to $e$ the excess has the same order $e^2$ as the Hilbert
function itself (Corollary~\ref{cor:sharp-discrepancy} and
Theorem~\ref{thm:leading-excess}).

For linear additional factors we also determine where the quotient ends.

\begin{maintheorem}\label{thm:intro-terminal}
For every integer $e\ge12$ and five general pairs of linear forms in
$S=k[x,y,z]$, let $I=(L_1^eM_1,\ldots,L_5^eM_5)$. Then
\[
 \dim_k(S/I)_{2e-2}=1,\qquad (S/I)_q=0\quad(q\ge2e-1),
 \qquad \operatorname{reg}(S/I)=2e-2.
\]
\end{maintheorem}

The surviving inverse class is the $(e-1)$st power of the conic through
the five coefficient points of the $L_i$; containment in the pure-power
ideal forces it, and the content of the theorem is that nothing else
survives. The proof (Theorem~\ref{thm:terminal}) deforms $M_i$ from $L_i$.
After removal of the forced conic factors, extension of an inverse class
through order six becomes a rank calculation of bounded size, settled by
a polynomial certificate for $e\ge36$ and by exact specializations for
$12\le e\le35$.

\subsection{Further results}
Two further results go in opposite directions from the transfer. First,
there is a partial converse
(Theorem~\ref{thm:polyharmonic-common-factors}): if all primary and
additional factors have the form $x+ty+t^2z$, so that their coefficient
points lie on one conic, and $\mathcal D=\partial_y^2-\partial_x\partial_z$,
then in an explicit range of degrees every relation whose components are
killed by $\mathcal D^{K+1}$ contains every additional factor in every
component, and so comes from a relation among the pure powers. Second,
the terminal class of Theorem~\ref{thm:intro-terminal} describes the
non-Lefschetz locus of six products in four variables
(Theorem~\ref{thm:non-lefschetz}): for
$A=k[x,y,z,w]/(L_i^eM_i)_{i=1}^6$ and $e\ge12$, the scheme where
$A_{2e-3}\xrightarrow{\,\ell\,}A_{2e-2}$ fails to be surjective contains
the classical Weddle quartic as a generically reduced component, the six
points $[M_i]$ as reduced isolated components, and at least six embedded
curves on the quartic with local ideal $(a^2,ab)$. Finally,
Appendix~\ref{sec:fixed-degrees} gives exact certificates for general
quintics and septics in four variables; with known results they establish
the equal-degree conjecture in four variables for $1\le d\le9$.

\subsection{Computer assistance}
The following results rely on exact computer calculations, all of which
are reproducible from the ancillary files
(Section~\ref{sec:reproducibility}): the $61$ modular ranks behind the
finite range of Theorem~\ref{thm:seven-intro} (of which the asymptotic
range uses three, together with seven rational rank-five checks),
the certificate in Theorem~\ref{thm:intro-terminal}, and the quintic and
septic certificates of Appendix~\ref{sec:fixed-degrees}.
Theorem~\ref{thm:non-lefschetz} uses Theorem~\ref{thm:intro-terminal}.
Theorems~\ref{thm:intro-obstruction}, \ref{thm:intro-uniform-hilbert}
and~\ref{thm:intro-inheritance}, Theorem~\ref{thm:substitution},
Corollary~\ref{cor:shapiro-defect} and
Theorem~\ref{thm:polyharmonic-common-factors} are proved without
computation.

\subsection{Related work}
Ideals generated by powers of general linear forms correspond, through
Macaulay duality, to fat points~\cite{EmsalemIarrobino1995}; their
Hilbert functions and the comparison with general forms are studied
in~\cite{Iarrobino1997,Chandler2007,Froberg2025,NagelTrok2019,BoijLundqvist2026}.
The weak Lefschetz property for such ideals is closely tied to unexpected
curves and hypersurfaces
\cite{SchenckSeceleanu2010,HarbourneSchenckSeceleanu2011,MMN2012,
CookHarbourneMiglioreNagel2018,BL2023}; see~\cite{MiglioreNagel2013} for
a survey. Ideals generated by products of linear forms attached to a
hyperplane arrangement, such as power ideals and zonotopal algebras,
have Hilbert series governed by the combinatorics of the arrangement
\cite{ArdilaPostnikov2010,HoltzRon2011,Berget2010}; there the products
are dictated by the arrangement, whereas here they are general of a
prescribed type. Frobenius dilation of Hilbert functions was developed
by Dinh and Ha~\cite{DinhHa2026} in connection with Demailly's
conjecture. Known cases of Fröberg's conjecture are recalled in
Appendix~\ref{sec:fixed-degrees}.

\subsection{Organization}
Sections~\ref{sec:seven}--\ref{sec:obstruction} treat four variables:
the proof of Theorem~\ref{thm:seven-intro}, the two critical ranks, the
substitution obstruction, and Theorem~\ref{thm:intro-obstruction}.
Sections~\ref{sec:inherited}--\ref{sec:terminal} treat five ternary
products: inherited relations and Theorem~\ref{thm:intro-inheritance},
the restriction sequence and the two critical degrees, the polyharmonic
converse, Theorem~\ref{thm:intro-uniform-hilbert}, and
Theorem~\ref{thm:intro-terminal}. Section~\ref{sec:non-lefschetz} returns
to four variables with the non-Lefschetz scheme, and
Section~\ref{sec:open} collects open problems. The appendices contain the
fixed-degree certificates and the basis used in the terminal calculation.

\section{Five powers and one split generator}\label{sec:seven}

In this section $S=k[x_0,x_1,x_2,x_3]$, $B_d=S/(L_1^d,\ldots,L_5^d)$
and $F_d=L_6^{\lceil d/2\rceil}D^{\lfloor d/2\rfloor}$.
Put $N(j)=\binom{j+3}{3}$ for $j\ge0$ and $N(j)=0$ otherwise.
Write $R_d(q)=[t^q](1-t^d)^6/(1-t)^4$, and denote its first
nonpositive index by $Q_d$. Thus the assertion about $I(d)$
in grade $q$ amounts to a rank assertion about
\begin{equation}\label{eq:seven-mult}
 F_d:(B_d)_{q-d}\longrightarrow(B_d)_q .
\end{equation}

\subsection{The finite rank test}
Take $y_i=L_i$ for $1\le i\le4$ as coordinates and let
$C=k[y_1,\ldots,y_4]/(y_1^d,\ldots,y_4^d)$.
Each coefficient of $L_5$ in this basis is nonzero. By the strong
Lefschetz theorem for monomial complete intersections
\cite{Stanley1980,Watanabe1987}, see also~\cite{HMMNWW2013},
multiplication by $L_5^d$ has maximal
rank on $C$. If $c_j=[t^j](1+t+\cdots+t^{d-1})^4$, it follows that
$b_j=\dim(B_d)_j=\max(0,c_j-c_{j-d})$.

Weak Lefschetz for quotients by powers of linear forms was studied by
Migliore, Mir\'o-Roig and Nagel~\cite{MMN2012}; Boij and
Lundqvist~\cite{BL2023} classified the case of $n+1$ equal powers
in $n$ variables. On $B_d$ a Lefschetz map is multiplication by one
general linear form. Our map~\eqref{eq:seven-mult} is multiplication by
a degree-$d$ product in two prescribed directions, so those
classification results do not determine its ranks.

\begin{lemma}\label{lem:seven-socle}
For $d\ge2$ the socle of $B_d$ vanishes in grades $j\le2d-3$.
\end{lemma}
\begin{proof}
Lift a socle element to $u\in C_j$ and write
$y_i u=L_5^d a_i$, with $a_i\in C_{j+1-d}$.
The equations $L_5^d(y_i a_h-y_h a_i)=0$ imply
$y_i a_h=y_h a_i$. Indeed the source degree $j+2-d$ of this
multiplication satisfies $2(j+2-d)+d\le4(d-1)$, so strong Lefschetz
gives injectivity. As $j+2-d<d$, these are polynomial identities
before taking the quotient. Unique factorization gives $a_i=y_i a$
for a common $a$ (take $a=0$ if all coefficients vanish).
Then $u-L_5^d a$ belongs to the socle of $C$, concentrated in degree
$4(d-1)$. Hence the original class is zero.
\end{proof}

\begin{proof}[Proof of Theorem~\ref{thm:seven-intro}, finite range]
In the $y$-coordinates the last three directions $L_5,L_6,6D$ have
coefficient vectors
\[
 (-1,4,-6,4),\quad(-4,15,-20,10),\quad(23,-33,21,-5).
\]
In grade $q$, take as rows all exponent vectors $u$ of weight $q$
with $0\le u_i<d$. Form two blocks of columns, each indexed by
vectors $v$ of weight $q-d$ with the same bounds. For the two blocks
take respectively $G=L_5^d$ and
$G=L_6^{\lceil d/2\rceil}(6D)^{\lfloor d/2\rfloor}$.
The entry $(u,v)$ is $[y^{u-v}]G$, interpreted as zero if an exponent
is negative. The corank (number of rows minus rank) of the resulting
matrix is the dimension of the quotient. Exact elimination over
$\mathbb F_{241}$ gives
\begin{align*}
 \operatorname{corank}_{241}(Q_d-1)&=R_d(Q_d-1),&
 \operatorname{corank}_{241}(Q_d)&=0 &&(1\le d\le30),\\
 \operatorname{corank}_{241}(70)&=0 &&&& (d=32).
\end{align*}
Table~\ref{tab:split-seeds} gives the three matrices used in the asymptotic
argument. Every entry of the full table can be checked with the source
in the ancillary material. For $d=1$ the four coordinate generators
already span $S_1$.
\begin{table}[htbp]
\centering\small
\caption{Ranks modulo $241$ after quotienting by the four coordinate powers.}
\label{tab:split-seeds}
\begin{tabular}{rrrrrr}\toprule
$d$&$q$&rows&columns&rank&corank\\\midrule
30&65&16708&16424&16368&340\\
30&66&16342&17606&16342&0\\
32&70&20060&20648&20060&0\\\bottomrule
\end{tabular}
\end{table}

A nonzero modular minor lifts to characteristic zero. At the lower
endpoint, $b_q-b_{q-d}=R_d(q)$ gives the opposite inequality on
corank, so \eqref{eq:seven-mult} is injective; at the upper endpoint
it is surjective. For clarity, the numerical bounds used here are
$Q_d-1-d\le2d-2$ and
$2(Q_d-1)-d\le4(d-1)$ for $2\le d\le30$.
They follow immediately by evaluating
$N(q)-6N(q-d)+15N(q-2d)$ up to its first zero or negative value.
The second inequality puts the five-power Hilbert function on its
untruncated branch through grade $Q_d-1$. If the kernel of
\eqref{eq:seven-mult} contained a nonzero element in a smaller grade,
Lemma~\ref{lem:seven-socle} would allow successive multiplication
by variables up to source grade $Q_d-1-d$, preserving a nonzero kernel
element. Zero in grade $Q_d$ persists in every higher grade.
This proves the finite assertion over $\mathbb Q$, hence over $k$.
\end{proof}

\subsection{Dilation and insertion of factors}
\begin{proof}[Proof of Theorem~\ref{thm:seven-intro}, asymptotic ranges]
The even degrees $30$ and $32$ play different roles. The former gives
control of relations up to grade $65$; for the latter we need only
vanishing in grade $70$. We use the Frobenius dilation method of
Dinh and Ha~\cite{DinhHa2026}. The additional step needed for
\eqref{eq:seven-intro} is to interpolate between prime multiples of a
fixed degree while retaining control of the first Koszul homology.

Write $H_i(g)_q$ for the total-degree $q$ part of the $i$th Koszul
homology of a tuple $g$. Suppose one row $g_j$ is replaced by $v g_j$,
where $v$ is linear and $e=\deg g_j$. Multiplying the wedge coordinates
containing $j$ by $v$ embeds the new complex into the old one.
The quotient is the Koszul complex of the other rows on the plane
$v=0$, shifted by one in homological degree and by $e$ in internal
degree. The homology sequence therefore gives an injection
\[
 H_1(g_1,\ldots,v g_j,\ldots,g_6)_q\hookrightarrow H_1(g)_q
\]
whenever the other five rows on the plane have no relation of total
degree $q-e$. This is the point at which the choice of directions matters.

Set $G=(L_1^2,\ldots,L_5^2,L_6 D)$. On any one of the seven factor
planes the five quadrics in the other rows are linearly independent.
This small rank calculation uses rational integers only: eliminate $x_0$
using the equation of the plane, and expand the five quadratic products
in the six quadratic monomials in $x_1,x_2,x_3$. The resulting matrices
all have rank five (the ancillary script \texttt{planes.py} constructs
them). Independence persists modulo all sufficiently large primes
$\ell$. On the plane in characteristic $\ell$, the Frobenius morphism
is flat~\cite{Kunz1969}; concretely, the monomials with exponents less
than $\ell$ form a basis over the subring of $\ell$th powers. The five $\ell$th powers of
these quadrics thus have no relation in total degree below $3\ell$.
Injectivity of the corresponding coefficient matrices lifts this
assertion, for each such $\ell$, to characteristic zero.

At degree $30$, $H_1(I(30))_q=0$ for $q<66$, where homology is taken
on the displayed six generators. Indeed, for the first five generators
the only possible positive homology is the kernel of $L_5^{30}$ on the
four-power complete intersection, shifted by $30$. It vanishes in this
range by strong Lefschetz. Adjoining $F_{30}$ gives the assertion by
injectivity in \eqref{eq:seven-mult}. These finitely many exactness
conditions persist at every sufficiently large prime $p$, by the ranks
of consecutive differentials. Flat Frobenius then gives zero first
homology below $66p$ for the list of $p$th powers, which is precisely
the tuple defining $I(30p)$. Exactness lifts to characteristic zero.

For a large $d$ choose the largest good prime $p\le d/30$, and the
largest good prime $\ell\le15p$ for the plane calculation. Both
$30p/d$ and $\ell/(15p)$ tend to one by the prime number theorem.
Obtain \eqref{eq:seven-intro} by adding the required linear factors
to the rows of $I(30p)$ one at a time. On the plane of an inserted
factor, the other five rows are nonzero multiples of the corresponding
$G_i^\ell$. A relation among the multiples would give a relation of
the same total degree among the powers. Since $e\ge30p$ and eventually
$30p+3\ell>66p$, every insertion preserves the vanishing of $H_1$
below $66p$.

In these grades higher Koszul homology also vanishes. For the five
pure powers only $H_1$ can occur, and it vanishes through $q$ whenever
$2q-d\le4(d-1)$, by the same four-power calculation. The exact
sequence adjoining the sixth generator proves the claim for the six
rows in the indicated range for large $d$. Euler characteristic now
gives the signed coefficient $R_d(q)$ below $66p$. It is positive
through $(11/5-\varepsilon)d$ for large $d$: up to $2d$ this follows
by comparing $N(q)$ and $6N(q-d)$; between $2d$ and $3d$ the leading
term of $6R_d(q)/d^3$ at $q/d=x$ is
\[
 P(x)=x^3-6(x-1)^3+15(x-2)^3 .
\]
Its first root in this interval is $\rho=2.24070\ldots>11/5$.
This proves the lower range.

For the upper range take the least sufficiently good prime $p\ge d/32$.
The grade-$70$ vanishing for $I(32)$ persists modulo $p$. In
characteristic $p$ the monomial basis over $p$th powers gives highest
possible nonzero grade $69p+4(p-1)$ for $S/I(32p)$.
Vanishing lifts to characteristic zero, and $I(32p)\subseteq I(d)$
by divisibility in each row. Thus $S/I(d)$ vanishes from grade
$73p-3$ onwards. Since $32p/d\to1$ and $\rho<73/32$, the predicted
series has already vanished throughout the asserted upper range.
\end{proof}

\begin{corollary}\label{cor:partition-lefschetz}
Let $A$ be the quotient by five general degree-$d$ forms in four
variables. Suppose $d\ge2$ and the parts of a partition
$\mu=(\mu_1,\ldots,\mu_r)$ of $d$ can be grouped into two sets
of sums $\lfloor d/2\rfloor$ and $\lceil d/2\rceil$.
For general linear forms $U_i$ set $G=\prod_i U_i^{\mu_i}$.
If $d\le30$, multiplication by $G:A_{q-d}\to A_q$ has maximal
rank for every $q$. For every $\varepsilon>0$ and all sufficiently
large $d$, the same conclusion holds in the two ranges of
Theorem~\ref{thm:seven-intro}.

In both assertions the first five forms may themselves be general
products of linear forms of any prescribed partitions of $d$. In
particular, six general products all of type $\mu$ have the predicted
Hilbert function in the same ranges.
\end{corollary}
\begin{proof}
The first four generators are a regular sequence on an open set
containing $(L_1^d,\ldots,L_4^d)$, so their quotient has Hilbert
function $c_j$ there. Multiplication by the fifth generator has
maximal rank at the power specialization, by strong Lefschetz on
the four-power complete intersection. The five-generator quotient
thus has constant dimensions $b_j$ on a neighborhood of the five
powers. On that neighborhood its graded pieces form vector bundles.
Specialize the $U_i$ in the two groups to $D$ and $L_6$,
respectively. The resulting map is~\eqref{eq:seven-mult}.
In the finite range its injectivity below $Q_d$ and surjectivity at
and above $Q_d$ were proved above. In the lower asymptotic range
vanishing of the first Koszul homology gives injectivity by the
sequence adjoining $F_d$; in the upper range the quotient vanishes.
Maximal rank in each of these grades is open in the joint parameters.
The four-generator complete intersection vanishes above $4(d-1)$,
so only finitely many open conditions need be intersected for any
fixed $d$. This gives the assertion for general parameters. We now use coalescence of factors, as in
Shapiro~\cite[Proposition~2.1]{Shapiro2026}. For prescribed
partitions on the first five forms, pull these open conditions back
to the affine space of their factors. The inverse image contains the
same witness: specialize all factors in row $i$ to $L_i$, $1\le i\le5$.
The resulting open set meets the dense open on which the factors in
each row are pairwise nonproportional. The dimensions $b_j$ are
constant there, and maximal rank of the sixth map gives the same
quotient dimensions as at the witness.
\end{proof}

\section{The two critical ranks}\label{sec:residues}

The finite computation suggests studying the change from $d$ to $d+1$.
The first five generators acquire the factors $v_i=L_i$; the sixth
acquires $v_6=L_6$ for even $d$ and $v_6=D$ for odd $d$. Denote the
old generators by $f_i$ and put $E=\{1,\ldots,6\}$. For $J\subseteq E$
write
\[
 M_J=S/(v_i f_i:i\in J,\ f_j:j\notin J),\qquad
 K_J=\ker\bigl((M_\varnothing)_s
       \xrightarrow{\prod_{i\in J}v_i}(M_J)_{s+|J|}\bigr).
\]
In particular $M_E=S/I(d+1)$. Four distinct directions among
$L_1,\ldots,L_6,D$ are independent: four curve points give a
Vandermonde determinant, and evaluation of the plane through three
curve points at $D$ gives $1-ijk\ne0$ for their distinct indices
$i,j,k$.

\begin{proposition}\label{prop:pair-glue}
If $s\ge2d-2$, each of the displayed multiplication maps is
surjective. For $|J|\ge2$ one has
\[
 K_J=\sum_{\{i,j\}\subseteq J}K_{\{i,j\}} .
\]
\end{proposition}
\begin{proof}
At the step adding $i$ the cokernel in grade $s+|J|$ is the
corresponding piece of $M_J/(v_i)$. At least four other pure powers
remain. Their dual points on the plane have no three collinear, and
their exponents $e_h$ sum to at most $4d+|J|-1$. By power-ideal
duality~\cite{EmsalemIarrobino1995}, a nonzero dual form of degree
$w=s+|J|$ would have multiplicities at least $w-e_h+1$ there.
Their sum is greater than $2w$. B\'ezout would make every smooth
conic of the pencil through these points a component, which is
impossible. This argument may be carried out over the algebraic closure.

Identify $(M_J)_{s+|J|}^*$ with its image $K_J^\perp$ in
$(M_\varnothing)_s^*$. For $m=|J|\ge3$, take an element common to
the images for $J\setminus\{i\}$, $i\in J$, and write $\eta_i$
for its unique preimages. The actions of $v_j$ on $\eta_i$ and
of $v_i$ on $\eta_j$ agree by injectivity on the common $(m-2)$-face.
Put $R=S/(f_h:h\notin J)$. These equations say that the tuple
$(\eta_i)$ is closed in the dual Koszul complex on $(v_i)_{i\in J}$
over $R$. Its obstruction to having a common primitive belongs to
$H_1((v_i)_{i\in J};R)_{s+m}^*$.

The old rows outside $J$ are a regular sequence: a prime containing
them must contain a linear factor from each, and any such selection
is independent. If $m=3$, compute Tor against $S/(v_i:i\in J)=k[z]$
using the resolution by the three old rows. Each restricts to a
nonzero multiple of $z^d$. A scalar change of basis gives the
Koszul complex on $(z^d,0,0)$; its first homology ends in degree
$2d-1$. For $m=4$ the same calculation over $k$ puts the first
homology in degree $d$. For $m>4$ the redundant linear parameters
add exterior generators of internal degree one. The first homology
therefore vanishes in grade $s+m$ in every case. A common primitive
$\Phi\in R_{s+m}^*$ exists. It annihilates $v_i f_i$ since
$v_i\Phi=\eta_i$ annihilates $f_i$. Thus $\Phi$ belongs to the dual
of the $J$-quotient. Taking intersections down to pairs and then
annihilators proves the formula.
\end{proof}

The pair subspaces overlap. We next identify the correction forced by
the points where three of the planes meet. For a graded ring $T$ and
forms $g$, use $H_i(g;T)_u$ for Koszul homology in total degree $u$.

\begin{proposition}\label{prop:twenty}
Let $n=d+1$ be sufficiently large, assume the predicted series at
degree $d$, and take $t=Q_n-1$ or $Q_n$, $s=t-6$. Put $N=\dim(M_\varnothing)_s$ and
\[
 a=-[z^{s+1}]\frac{(1-z^d)^5}{(1-z)^3},\qquad
 b=8d-3(s+2)-3 .
\]
There is a linear map from a space of dimension $6a+15b-20$ to a space
of dimension $N$ whose kernel and cokernel are $H_1(I(n))_t$ and
$(S/I(n))_t$.
Its six plane summands before taking relations have dimension $a$
each. After their images have been taken, the remaining map has source
dimension $15b-20$ and is induced by relations on the fifteen lines.
\end{proposition}
\begin{proof}
The first cutoff satisfies $Q_d=\lceil\rho d-2+O(d^{-1})\rceil$:
substitution of $x=(q+2)/d$ in the three contributing binomial
coefficients gives $6R_d(q)=d^3P(x)-d[x-6(x-1)+15(x-2)]$,
and the root $\rho$ is simple. Thus $s<Q_d$ for large $d$, and
$(s+j)/d\to\rho$ for fixed $j$. In particular $2d\le s+2<s+6<3d$.
The five-power strong Lefschetz calculation of
Section~\ref{sec:seven} rules out positive homology on those five
rows up to $s$. The assumed Hilbert function at $d$ gives
injectivity of $F_d$ through $s$, so the old six-row complex has
only degree-zero homology at $s$.

Replace the Koszul complex on each $v_i f_i$ by the two-term complex
with differential
\[
 \begin{pmatrix} f_i&-1\\0&v_i\end{pmatrix}:
 S(-d-1)\oplus S(-1)\longrightarrow S(-1)\oplus S .
\]
Cancelling the unit recovers the original complex. Filter the tensor
product by the number of right-hand blocks. On the summand indexed by
$J$, the diagonal differential is Koszul on the $v_i$ in $J$ and
the old $f_j$ outside $J$, in grade $s+|J|$. Off-diagonal terms
remove an index. The following list gives all nonzero diagonal homology
in the indicated grade; the dimensions are per subset.
\[
\begin{array}{c|c|c|c}
 |J|&\text{number}&\text{homological degree}&\text{dimension}\\ \hline
 0&1&0&N\\ 1&6&1&a\\ 2&15&1&b\\ 3&20&2&1
\end{array}
\]

To verify the list, quotient first by the independent linear forms for
$|J|\le4$. On a plane three of the surviving pure powers form a
complete intersection. On its remaining two rows the top Koszul
homology in grade $w=s+1$ is dual to the final quotient in grade
$5d-3-w$, by the socle pairing of the complete intersection. Four
projected pure powers force the quotient to vanish above $2d-2$ by
the conic argument above. Both $w$ and its reflected grade exceed this
bound. Only $H_1$ remains, with dimension $a$ by Euler characteristic.
On a line two independent pure powers give the same argument with
reflected grade $4d-2-w$, now $w=s+2$. Three distinct binary pure
powers already give zero quotient above $\lfloor3(d-1)/2\rfloor$,
by the three multiplicity conditions on a binary dual form. Since
$2<\rho<5/2$, both grades exceed this bound. Euler characteristic
gives $b=-[z^w](1-z^d)^4/(1-z)^2$.
At a triple point the calculation on $(z^d,0,0)$ gives one copy of
$(k[z]/z^d)(-2d)$ in homological degree two in the required range.
Four linear forms leave $k$ and only the internal degrees $0,d,2d$.
Five or six linear forms leave exterior classes from the redundant
linear forms, of total internal degree at most $d+1$ or $2$,
respectively. None meets the indicated grade.

Cancel contractible pairs in each diagonal complex over $k$. Each
resulting differential goes to a proper subset: elimination composes
the off-diagonal maps along chains of subsets. This leaves a
three-term complex with dimensions $20,6a+15b,N$. The map from a
single plane summand to the empty summand has image $K_{\{i\}}$:
restrict the filtration to these two faces and take degree-zero
homology, where the inclusion acts by $v_i$.

It remains to check that the twenty relations are independent even
after projection to their line components. Fix $i,j$. For $k$ outside
this pair, the incident triple maps to the line by the connecting
homomorphism for the displayed two-block complex of row $k$ on that
line, with the three other old rows present. Its middle complex
computes the binary homology of four forms with one exponent raised.
Its $H_2$ at local grade $w$ is dual to the quotient in grade
$4d-1-w$ (use two coprime pure powers and the socle pairing).
At least three distinct pure powers survive, with sum of exponents
at most $3d+1$; their quotient vanishes above
$\lfloor(3d-2)/2\rfloor$. Hence the connecting map is injective for
$w=s+3$ and $w=s+6$. The incident class is killed by $v_k$.
Multiplying it by the other three directions outside the pair and
different from $k$ gives a nonzero class at the later grade: on
$k[z]/z^d$ these multipliers are nonzero scalar multiples of $z$,
and $s+6<3d$. The same product kills the other three incident classes.
Thus all four classes on this line are independent. Projecting a
putative relation to each line proves independence of the twenty
triple images. Dividing out these images proves both assertions.
\end{proof}

This construction gives a test for the conjectural equality in an
individual degree. In the lower grade the plane map and the induced
line map must both be injective. In the upper grade the line map must
surject onto the cokernel left by the planes. These criteria follow
from Proposition~\ref{prop:twenty} by Euler characteristic, since the
signed dimension difference is $R_n(t)$. The lower grade then
propagates by Lemma~\ref{lem:seven-socle}, and the upper grade by
generation. In particular, verifying the ranks on each plane separately
leaves a genuine question about their images in a common space.

Both maps admit a description by polynomial division. For a relation
among the five forms on $v_i=0$, lift its coefficients and write
$\sum_{k\ne i}h_k f_k=v_i u$. Its image in $(M_\varnothing)_s$
is the class of $u$. Write $U$ for the span of the six plane images.
For a relation on $v_i=v_j=0$, choose lifts with
\[
 \sum_{k\notin\{i,j\}}h_k f_k+v_i a_i+v_j a_j=0.
\]
The plane quotients vanish in degree $s+1$, so we can write
$a_i=\sum_{k\ne j}\alpha_k f_k+v_j p$ and
$a_j=\sum_{k\ne i}\beta_k f_k+v_i q$. The line map takes the
relation to $[p+q]$ modulo $U$, up to the orientation of the edge.
Indeed $v_i v_j(p+q)$ lies in the ideal defining $M_{\{i,j\}}$.
Changing a plane decomposition shifts the output by a single-plane
kernel; changing the split of $v_i a_i+v_j a_j$ changes $p$ and $q$
by opposite amounts. Terms added to a lift by multiples of $v_i$ or
$v_j$ are absorbed into the plane decompositions, and a Koszul boundary
lifts to a polynomial identity. This gives the line component after
the plane images have been taken. Its source dimension $15b-20$ grows
only linearly with $d$.

\section{An obstruction to a common substitution}\label{sec:substitution}

A successful tuple in one degree suggests a way to construct one in
higher degree. Substitute the same four forms into each member of the
tuple, then take linear combinations with polynomial coefficients to
adjust the degree. The following estimate limits this construction,
even when the original tuple and the substitution are chosen anew at
every degree. Retain the notation $R_d(q)$ and $\rho$ of
Section~\ref{sec:seven}.

\begin{theorem}\label{thm:substitution}
Let $f_1,\ldots,f_6$ be forms of degree $D\ge1$ in four variables
over a field of characteristic zero. Let $A_0,\ldots,A_3$ be forms
of degree $m\ge32$ in $S=k[x_0,\ldots,x_3]$, and take any six forms
$G_i$ of degree $d=mD+a$ in the ideal $(f_j(A):1\le j\le6)$.
If $0\le a\le\lfloor m/4\rfloor$, then at the first zero $Q_d$ of the predicted series,
\[
 \dim_k(S/(G_1,\ldots,G_6))_{Q_d}
 \ \ge\ \binom{\lfloor m/4\rfloor+3}{3}.
\]
More generally, for each $0\le\theta<1/\rho$, the same assertion with
lower bound $c_\theta m^3$ holds for $0\le a\le\theta m$ and all
sufficiently large $m$, uniformly in $D$.
\end{theorem}
\begin{proof}
Write $T=k[y_0,\ldots,y_3]$, $b=\lfloor\rho D\rfloor-2$, and
$x=(b+2)/D$. Then $2\le x\le\rho<9/4$. The cubic coefficient
identity in Section~\ref{sec:residues} gives
$6R_D(b)=D^3 P(x)+D(24-10x)>0$. This identity remains valid at
$b=2D-2,2D-1$, where the cubic expression for the last binomial
coefficient is zero.

We claim $\dim(T/(f))_b\ge R_D(b)$ for every seed tuple. If $b<2D$
it suffices to count the multiples of the generators. Otherwise write
$h=b-2D$ and $u=\lfloor D/2\rfloor$, so $0\le h<u$. For the four
coordinate powers together with $y_0^uy_1^{D-u}$ and
$y_2^uy_3^{D-u}$, every pairwise least common multiple has degree
at least $D+u$. Thus the six forms have no relation with coefficients
of degree $h$. This injectivity holds for general forms. Their fifteen
Koszul pairs have independent degree-$h$ multiples: in each component
of a dependence the coefficients would give just such a relation.
Counting these multiples in the kernel gives $R_D(b)$ as a lower
bound; semicontinuity extends it to every tuple.

For a regular substitution, lift a homogeneous basis of $S/(A)$ to
$S$. Its lifts generate over $k[A_0,\ldots,A_3]$ by induction on
degree. Finiteness and dimension give algebraic independence of the
$A_i$, and comparison of Hilbert series proves that the lifts form a
free basis. Hence
\[
 \HS_{S/(f(A))}(t)=(1+t+\cdots+t^{m-1})^4\HS_{T/(f)}(t^m).
\]
Write $c_m(j)=[t^j](1+t+\cdots+t^{m-1})^4$. At grade $q$ the
seed term alone contributes at least $R_D(b)c_m(q-bm)$. For fixed $f$
this lower bound extends to arbitrary substitutions by upper
semicontinuity of quotient dimension. It also holds for the $G_i$,
since their quotient surjects onto $S/(f(A))$.

To locate the first zero, put $p=\lfloor\rho d\rfloor$. Then
\[
                       p-1\le Q_d\le p .
\]
Indeed $R_d(q)>0$ for $q<d$. For $d\le q\le 2d-3$ the ratio
$N(q)/N(q-d)$ exceeds eight, since each of its three factors
$(q+i)/(q-d+i)$, $1\le i\le3$, is at least two and the first
two are strictly larger. From $2d-2$ through $p-2$ the cubic
identity gives positivity, since $2\le(q+2)/d\le\rho$. At $p$,
set $z=(p+2)/d$. Then $\rho+1/d<z\le5/2$, and
$P'(x)\le-6$ on $[2,5/2]$: its derivative polynomial $P'$ is
convex and both endpoint values are at most $-6$. Consequently
$6R_d(p)<-6d^2+4d<0$.

By symmetry $c_m(Q_d-bm)=c_m(L)$ with $L=(b+4)m-4-Q_d$.
The two bounds for $Q_d$ and the definition of $b$ give
\[
                       m-4-\rho a < L \le 2m-3.
\]
For the upper inequality use $m\lfloor\rho D\rfloor\le p$.
The coefficients $c_m(j)$ are nondecreasing up to their middle
index $2m-2$. One way to see this is to use the strong Lefschetz
theorem on the complete intersection of four coordinate $m$th powers.
If $a\le m/4$ and $m\ge32$, then
$m-4-\rho a>7m/16-4\ge m/4$. Thus
$c_m(L)\ge c_m(\lfloor m/4\rfloor)=\binom{\lfloor m/4\rfloor+3}{3}$.
For $a\le\theta m$ choose $0<\eta<1-\rho\theta$. For all sufficiently
large $m$ we have $L\ge\lfloor\eta m\rfloor$, giving the second
assertion with a constant independent of $D$.
\end{proof}

For example, $D=13$, $m=32$, $a=8$ give $d=424$, $b=27$ and
$R_D(b)=40$. The first zero of the predicted series is $Q_{424}=949$,
where $L=39$ and $c_{32}(39)=\binom{42}{3}-4\binom{10}{3}=11000$. Hence
every such substitution and mixing of its rows leaves at least
$440000$ dimensions in grade $949$, far above the bound
$\binom{11}{3}=165$ of the theorem. The use of the same substitution on
all six rows is essential to the proof. In~\eqref{eq:seven-intro}, the factors can instead be
inserted one row at a time.

\section{Linear factors of high multiplicity}\label{sec:obstruction}

Throughout this section $k$ has characteristic zero and
$S=k[x_1,x_2,x_3,x_4]$.  Let $D=k[X_1,X_2,X_3,X_4]$, on which $S$ acts by
differentiation, $x_i\circ G=\partial G/\partial X_i$.  For each $j$ the
pairing $S_j\times D_j\to k$, $(f,G)\mapsto f\circ G$, is perfect, since
$x^\alpha\circ X^\beta=\alpha!\,\delta_{\alpha\beta}$ when
$|\alpha|=|\beta|$.  This is the inverse-system framework of Emsalem and
Iarrobino~\cite{EmsalemIarrobino1995}; we need only the elementary
direction recorded in Lemma~\ref{lem:apolarity}.  For $a\in k^4$ write $\ell_a=a_1x_1+\cdots+a_4x_4$.
For $G\in D$ let $\operatorname{ord}_aG$ be the smallest degree of a
nonzero homogeneous component of $G(a+Z)$ as a polynomial in $Z$, with
$\operatorname{ord}_a0=\infty$.  Then
$\operatorname{ord}_a(GH)\ge\operatorname{ord}_aG+\operatorname{ord}_aH$.

\begin{lemma}[Apolarity]\label{lem:apolarity}
Let $a\in k^4$, let $0\le e\le j$, and let $G\in D_j$ satisfy
$\operatorname{ord}_aG\ge j-e+1$.  Then $\ell_a^{\,e}\circ G=0$.
\end{lemma}

\begin{proof}
Write $G(a+Z)=\sum_{p}G_p(Z)$ with $G_p$ homogeneous of degree $p$, so
$G_p=0$ for $p\le j-e$.  As polynomials in $s$ and $Y$,
\[
 \sum_{m=0}^{j}\frac{s^m}{m!}\,(\ell_a^{\,m}\circ G)(Y)
 =G(Y+sa)=\sum_{p=0}^{j}s^{\,j-p}G_p(Y);
\]
the first equality is Taylor's formula, and the second follows from the
homogeneity of $G$ for $s\ne0$, hence identically.  Comparing coefficients
of $s^e$ gives $(\ell_a^{\,e}\circ G)/e!=G_{j-e}=0$.
\end{proof}

\begin{lemma}[Six quadric cones]\label{lem:cones}
Let $\ell_1,\ldots,\ell_6\in S_1$ be arbitrary, let $\lambda\ge1$, and let
$e\ge5\lambda+1$.  Then
\[
 \dim_k\bigl(S/(\ell_1^{\,e},\ldots,\ell_6^{\,e})\bigr)_{12\lambda}\ge1.
\]
\end{lemma}

\begin{proof}
Put $J=(\ell_1^{\,e},\ldots,\ell_6^{\,e})$ and $j=12\lambda$.  If $e>j$
then $J_j=0$, so assume $e\le j$.  Write $\ell_i=\ell_{a^{(i)}}$ with
$a^{(i)}\in k^4$.  Since the pairing is perfect,
$\dim_k(S/J)_j$ is the dimension of the space of $G\in D_j$ with
$f\circ G=0$ for all $f\in J_j$.  As $J_j$ is spanned by the products
$\ell_i^{\,e}h$ with $h\in S_{j-e}$, and
$(\ell_i^{\,e}h)\circ G=h\circ(\ell_i^{\,e}\circ G)$, it suffices to find
a nonzero $P\in D_j$ with $\ell_i^{\,e}\circ P=0$ for every $i$.

For each $i$ choose a nonzero $Q_i\in D_2$ that vanishes at the five
points $a^{(m)}$ with $m\ne i$ and whose four first partial derivatives
vanish at $a^{(i)}$.  These are at most nine linear conditions on the
ten-dimensional space $D_2$.  By Euler's formula $Q_i(a^{(i)})=0$ as well,
so $\operatorname{ord}_{a^{(i)}}Q_i\ge2$ and
$\operatorname{ord}_{a^{(m)}}Q_i\ge1$ for $m\ne i$.  Let
$P=(Q_1\cdots Q_6)^\lambda$.  It is nonzero of degree $12\lambda$, and
$\operatorname{ord}_{a^{(i)}}P\ge7\lambda\ge j-e+1$ for every $i$.
Lemma~\ref{lem:apolarity} gives $\ell_i^{\,e}\circ P=0$.
\end{proof}

When the $\ell_i$ are general they correspond to six general points of
$\mathbb P^3$, and $Q_i$ is the quadric cone with vertex at the $i$th point
through the other five.  For $e=5\lambda+1$ the bound is then an equality
by~\cite[Proposition~7.3]{BoijLundqvist2026}, because $5\cdot12\lambda=12(e-1)$
is the endpoint of the second range there.  Lemma~\ref{lem:cones} needs
neither that computation nor any generality assumption.

\begin{lemma}\label{lem:prediction-vanishes}
Let $d\ge4$.  The series $[(1-t^d)^6/(1-t)^4]_+$ has no terms of degree at
least $9d/4$.
\end{lemma}

\begin{proof}
Let $c_j$ be the coefficient of $t^j$ in $(1-t^d)^6/(1-t)^4$.  It suffices
to show $c_j<0$ for $j=\lceil9d/4\rceil$.  Since $d\ge4$, this $j$
satisfies $2d\le j\le3d-2$.  In this range
$c_j=\binom{j+3}{3}-6\binom{j-d+3}{3}+15\binom{j-2d+3}{3}$, and each
binomial coefficient equals the corresponding cubic polynomial.  With
$u=j+2$ and $(u-c)^3-(u-c)=(u-c-1)(u-c)(u-c+1)$, expansion gives
\[
 6c_j=d^3\,\Gamma(u/d)-10u+24d,\qquad
 \Gamma(x)=10x^3-72x^2+162x-114.
\]
Here $\Gamma'(x)=6(5x-9)(x-3)\le0$ on $[9/5,3]$ and $\Gamma(9/4)=-3/32$.
Since $9/4<u/d\le3$, we get $\Gamma(u/d)\le-3/32$ and $-10u+24d<3d/2$.
Hence $6c_j<-\tfrac{3d}{32}(d^2-16)\le0$.  (For $d=4s$ and $j=9s$ one
gets $c_j=-s^3-54s^2-11s+10$.)
\end{proof}

\begin{theorem}\label{thm:factor-obstruction}
Let $d\ge4$ and $\lambda=\lceil3d/16\rceil$.  Let $\ell_1,\ldots,\ell_6\in
S_1$ and let $F_1,\ldots,F_6\in S_d$ with $F_i\in\ell_i^{\,5\lambda+1}S$
for each $i$.  Then
\[
 \dim_k\bigl(S/(F_1,\ldots,F_6)\bigr)_{12\lambda}\ge1,
\]
whereas the coefficient of $t^{12\lambda}$ in
$[(1-t^d)^6/(1-t)^4]_+$ is zero.  In particular, the conclusion holds if
each $F_i$ is divisible by $\ell_i^{\,e}$ for some $e\ge(15d+91)/16$.
\end{theorem}

\begin{proof}
Put $e=5\lambda+1$.  The ideal $I=(F_1,\ldots,F_6)$ is contained in
$J=(\ell_1^{\,e},\ldots,\ell_6^{\,e})$, so $S/I$ maps onto $S/J$, and
Lemma~\ref{lem:cones} gives the first assertion.  Since
$12\lambda\ge9d/4$, the second follows from
Lemma~\ref{lem:prediction-vanishes}.  Finally
$\lambda\le(3d+15)/16$, so $5\lambda+1\le(15d+91)/16$.
\end{proof}

\begin{remark}\label{rem:six-powers}
For powers of linear forms the theorem is far from sharp.  Let $d\ge4$ and let
$\ell_1,\ldots,\ell_6$ be general.  By~\cite[Proposition~7.3]{BoijLundqvist2026}
the quotient by $(\ell_1^{\,d},\ldots,\ell_6^{\,d})$ has dimension
$\binom{2d-4}{3}-6\binom{d-2}{3}$ in degree $2d-1$.  The predicted
coefficient there is
$\binom{2d+2}{3}-6\binom{d+2}{3}=d(d+1)(d-4)/3\ge0$, and all earlier
predicted coefficients are positive.  The difference is exactly $4$.  The
point of Theorem~\ref{thm:factor-obstruction} is that it applies to every
perturbation of the six powers that retains a linear factor of high
multiplicity.
\end{remark}

\subsection{Taylor coefficients of the Rogers--Szegő polynomial}

For $m\geq1$ write $[m]_q=1+q+\cdots+q^{m-1}$ and
$[m]_q!=\prod_{i=1}^m[i]_q$, with $[0]_q!=1$.
For $\alpha\in\mathbb Z_{\geq0}^4$ with $|\alpha|=d$, define the Gaussian
multinomial coefficient and the homogeneous Rogers--Szegő polynomial by
\[
 \genfrac{[}{]}{0pt}{}{d}{\alpha}_q=\frac{[d]_q!}{\prod_{i=1}^4[\alpha_i]_q!},
 \qquad H_d(y;q)=\sum_{|\alpha|=d}\genfrac{[}{]}{0pt}{}{d}{\alpha}_q y^\alpha;
\]
see~\cite[Chapter~3]{Andrews1976} and~\cite[Section~1]{CameronVinroot2013}.
Let $M=y_1+\cdots+y_4$.  Then $H_d(y;1)=M^d$, so $H_d(a_1x_1,\ldots,a_4x_4;q)$
is a deformation of the power $\ell_a^{\,d}$.  One may ask whether it
completes $x_1^d,\ldots,x_4^d,(x_1+\cdots+x_4)^d$ to six forms with the
predicted series.  Its Taylor coefficients at $q=1$ retain a high power
of $M$.

\begin{lemma}\label{lem:jet-factor}
Put $Q_{d,\nu}(y)=[\epsilon^\nu]H_d(y;1+\epsilon)$.
If $d\geq2\nu$, then $M^{d-2\nu}$ divides $Q_{d,\nu}$.
\end{lemma}

\begin{proof}
For $\nu=0$, the assertion is $Q_{d,0}=M^d$.  For $\nu\ge1$,
we work with formal power series in $\epsilon$ over $\mathbb Q$.  For an
integer $b\ge1$,
\[
 \frac{[b]_{1+\epsilon}}{b}
 =\frac{(1+\epsilon)^b-1}{b\,\epsilon}
 =\sum_{i\geq0}\frac{1}{b}\binom{b}{i+1}\epsilon^i .
\]
The constant term is $1$, and for $i\ge1$ the coefficient of
$\epsilon^i$ is a polynomial in $b$ of degree $i$.  Hence the coefficient
of $\epsilon^i$ in $\log([b]_{1+\epsilon}/b)$ is a polynomial in $b$ of
degree at most $i$, since every product of coefficients contributing to it
has orders summing to $i$.  Summing over $1\le b\le m$, the coefficient
of $\epsilon^i$ in $\log([m]_{1+\epsilon}!/m!)$ is a polynomial in $m$ of
degree at most $i+1$; this also holds for $m=0$.

For fixed $d$ it follows that the coefficient of $\epsilon^i$ in
$\log\bigl(\genfrac{[}{]}{0pt}{}{d}{\alpha}_{1+\epsilon}\big/\binom
d\alpha\bigr)$ is a polynomial in $\alpha$ of total degree at most $i+1$.
After exponentiation, the coefficient of $\epsilon^\nu$ is a polynomial
$p_{d,\nu}(\alpha)$ of degree at most $2\nu$: a product of logarithmic
coefficients of orders $i_1,\ldots,i_h\geq1$ with $\sum i_c=\nu$ has degree
at most $\sum_c(i_c+1)=\nu+h\leq2\nu$.

Let $E_i=y_i\partial/\partial y_i$.  Since $E_iy^\alpha=\alpha_iy^\alpha$,
\[
 Q_{d,\nu}=p_{d,\nu}(E_1,E_2,E_3,E_4)\,M^d .
\]
Moving all derivatives to the right expresses this operator as a
differential operator of order at most $2\nu$ with polynomial
coefficients.  Every derivative of $M^d$ of order at most $2\nu$ is a
scalar multiple of $M^{d-c}$ with $c\le2\nu$, hence divisible by
$M^{d-2\nu}$.
\end{proof}

For $\nu=1$ and $d\ge2$ one has $p_{d,1}(\alpha)=\frac12\bigl(\binom d2-\sum_i\binom{\alpha_i}2\bigr)$,
and therefore $Q_{d,1}=\binom d2M^{d-2}e_2(y)$, where
$e_2(y)=\sum_{i<j}y_iy_j$ is not divisible by $M$.  Thus the divisibility bound is sharp for $\nu=1$.

\begin{corollary}\label{cor:fixed-jet}
Fix $K\geq0$ and put
$T_{d,K}(y;\epsilon)=\sum_{\nu=0}^K\epsilon^\nu Q_{d,\nu}(y)$.  Suppose
that $d=32K+16$ or $d\ge32K+91$.  Let $\ell_1,\ldots,\ell_5\in S_1$, let
$a\in k^4$, and let $\epsilon\in k$.  Then the ideal
\[
 \bigl(\ell_1^{\,d},\ldots,\ell_5^{\,d},\;
 T_{d,K}(a_1x_1,\ldots,a_4x_4;\epsilon)\bigr)
\]
does not have the Hilbert series $[(1-t^d)^6/(1-t)^4]_+$.  The same holds
over the rational function field $k(a_1,\ldots,a_4,\epsilon)$.
\end{corollary}

\begin{proof}
Let $\lambda=\lceil3d/16\rceil$.  If $d=32K+16$ then $\lambda=6K+3$ and
$5\lambda+1=d-2K$.  If $d\ge32K+91$ then
$5\lambda+1\le(15d+91)/16\le d-2K$.  In both cases $d\ge2K$, so
Lemma~\ref{lem:jet-factor} shows that every $Q_{d,\nu}$ with $\nu\le K$ is
divisible by $M^{d-2K}$.  The substitution $y_i\mapsto a_ix_i$ sends $M$
to $\ell_a$, so the sixth generator is divisible by
$\ell_a^{\,5\lambda+1}$.  The first five are divisible by
$\ell_i^{\,5\lambda+1}$ because $5\lambda+1\le d$.  Apply
Theorem~\ref{thm:factor-obstruction}, over $k$ or over
$k(a_1,\ldots,a_4,\epsilon)$.
\end{proof}

The corollary concerns truncated generators only; it says nothing about
the full polynomial $H_d$ or about truncations whose order grows with $d$.

\begin{question}\label{q:rogers-szego}
For general $a\in k^4$ and general $q\in k$, does
\[
 \bigl(x_1^d,x_2^d,x_3^d,x_4^d,(x_1+\cdots+x_4)^d,
 H_d(a_1x_1,\ldots,a_4x_4;q)\bigr)
\]
have the Hilbert series $[(1-t^d)^6/(1-t)^4]_+$ for every $d$?  More
generally, is there an explicit sequence of six forms in four variables
with the predicted series in every degree?
\end{question}

For the last question, the balanced product in~\eqref{eq:seven-intro}
gives a test case with the same directions in every degree.
Its smaller exponent grows linearly with $d$, so the fixed-order
obstruction above does not apply. The finite verification in Theorem~\ref{thm:seven-intro} reaches
generator degree $30$. For a uniform proof, the decisive
issue is whether the plane and line images described in
Section~\ref{sec:residues} have the required joint ranks on either
side of the predicted first zero.

\section{Relations that survive additional factors}\label{sec:inherited}
The containment
\[
 (L_1^{d-1}M_1,\ldots,L_r^{d-1}M_r)
 \subseteq (L_1^{d-1},\ldots,L_r^{d-1})
\]
gives a lower bound for the Hilbert function. There is a second comparison:
a relation among the pure powers produces a relation among the products.
For five generators in three variables, this second comparison gives a
strictly larger bound in a range where the predicted Hilbert function is
positive.

\subsection{Transfer of relations}
For homogeneous forms $f_1,\ldots,f_r$ of degree $e$ in a polynomial
ring $S$, write
\[
 \mathcal Z_p(f)=
 \left\{(A_1,\ldots,A_r)\in S_{p-e}^{\,r}:
                 \sum_{i=1}^r f_iA_i=0\right\}.
\]
The subscript is the total degree of the relation.

\begin{lemma}\label{lem:relation-transfer}
Let $H_1,\ldots,H_r$ be nonzero forms of degree $s$. The map
\[
 \mathcal Z_p(f)\longrightarrow\mathcal Z_{p+rs}(f_1H_1,\ldots,f_rH_r),
 \qquad
 (A_i)_{i=1}^r\longmapsto
 \left(A_i\prod_{h\ne i}H_h\right)_{i=1}^r
\]
is injective.
\end{lemma}
\begin{proof}
The image is a relation because its contraction with $(f_iH_i)_i$ is
$(\prod_h H_h)\sum_i f_iA_i$. Each coefficient gains degree $(r-1)s$,
so the total degree increases by $rs$. Multiplication by a nonzero
polynomial is injective in $S$, which proves injectivity.
\end{proof}

In particular, if $S=k[x,y,z]$, $e=d-1$, $f_i=L_i^e$ and
$I=(f_1M_1,\ldots,f_5M_5)$, the multiplication map
$S_{q-d}^{\,5}\to S_q$ gives
\begin{equation}\label{eq:inherited-hf-bound}
 \dim_k(S/I)_q\ge
 \binom{q+2}{2}-5\binom{q-d+2}{2}+\dim_k\mathcal Z_{q-5}(f).
\end{equation}
We shall use this inequality with $q\ge d$.

\subsection{Five points on a conic}
For five distinct points $Z=\{P_1,\ldots,P_5\}$ on a smooth conic
$C\subset\mathbb P^2$, let $h_Z(a;m)$ be the dimension of degree-$a$
forms vanishing to order at least $m$ at every $P_i$.
The following calculation is classical (compare~\cite{Harbourne1986});
we include its short proof to make the relation count explicit.

\begin{lemma}\label{lem:conic-postulation}
For every integer $w\ge0$,
\[
 h_Z(5w;2w)=\frac{5w^2+5w+2}{2}.
\]
If $a\ge2$, $m\ge1$, and $2a-5m<0$, multiplication by an equation of $C$ identifies the
spaces counted by $h_Z(a-2;m-1)$ and $h_Z(a;m)$.
\end{lemma}
\begin{proof}
A degree-$a$ form not divisible by the equation of $C$ restricts to
a nonzero section of degree $2a$ on $C$. Vanishing to order $m$ at the
five points is impossible when $2a<5m$. Since $C$ is smooth at the
points, dividing by its equation reduces each required multiplicity by one.

For the first assertion, use the residual sequence
\[
 0\longrightarrow\mathcal I_{(m-1)Z}(a-2)
 \longrightarrow\mathcal I_{mZ}(a)
 \longrightarrow\mathcal O_C(a)(-mP_1-\cdots-mP_5)
 \longrightarrow0.
\]
Starting at $(a,m)=(5w,2w)$, the restriction term after $h$ residual
steps has degree $h$ on $C\simeq\mathbb P^1$.
After $2w$ steps the residual sheaf is $\mathcal O_{\mathbb P^2}(w)$.
Its first cohomology vanishes, as does that of every restriction term.
Induction therefore gives
$H^1(\mathbb P^2,\mathcal I_{2wZ}(5w))=0$.
The fat-point scheme $2wZ$ has length $5\binom{2w+1}{2}$, and hence
\[
 h_Z(5w;2w)=\binom{5w+2}{2}-5\binom{2w+1}{2}
           =\frac{5w^2+5w+2}{2}.
 \qedhere
\]
\end{proof}

For the rest of this section, $k$ has characteristic zero and the
$L_i$ are general linear forms in $k[x,y,z]$. We may extend scalars
to an algebraic closure without changing the Hilbert function.
Their coefficient points lie on a smooth conic. If
$J_e=(L_1^e,\ldots,L_5^e)$, inverse-system duality gives
\begin{equation}\label{eq:ternary-apolarity}
 \dim_k(S/J_e)_a=h_Z(a;a-e+1)\qquad(a\ge e);
\end{equation}
see~\cite{EmsalemIarrobino1995}, and~\cite{GeramitaSchenck1998,IarrobinoKanev1999}
for inverse systems of fat points more generally.
For this special case the identification is immediate from the
differential pairing: after taking $L_i=x$, the condition
$\partial_X^e F=0$ says that every degree-$a$ monomial of $F$ has
$X$-exponent less than $e$. This is exactly vanishing to order
$a-e+1$ at $[1:0:0]$ in the dual plane.

\subsection{A strict excess over both comparison bounds}
Put
\begin{align}
 G(u,v)&=\frac{5u^2-10uv-4v^2+5u-2v+2}{2},\label{eq:ternary-G}\\
 H(u,v)&=\frac{5(u-v-2)^2+5(u-v-2)+2}{2}.\label{eq:ternary-H}
\end{align}
The two expressions are, respectively, the predicted coefficient and
the pure-power Hilbert function in the following theorem.

\begin{theorem}\label{thm:ternary-family}
Let $u\ge12$ and $4\le v\le\lfloor u/4\rfloor+1$ be integers. Set
$d=3u+1$ and $q=5u+v$. For general linear forms $L_1,\ldots,L_5$
and any nonzero linear forms $M_1,\ldots,M_5$ in $S=k[x,y,z]$, let
$I=(L_1^{3u}M_1,\ldots,L_5^{3u}M_5)$. Then
\begin{equation}\label{eq:ternary-strict-bound}
 \dim_k(S/I)_q\ge G(u,v)+\binom{3v-10}{2}
                 >\max\{G(u,v),H(u,v)\}.
\end{equation}
Moreover,
\[
 [t^q]\trunc{\frac{(1-t^d)^5}{(1-t)^3}}=G(u,v),
 \qquad h_Z(q;q-d+2)=H(u,v).
\]
\end{theorem}
\begin{proof}
Write $e=3u$ and $f=(L_1^e,\ldots,L_5^e)$.
At degree $q$, the multiplicity in~\eqref{eq:ternary-apolarity} is
$m=2u+v+1$, so $2q-5m=-(3v+5)$.
By Lemma~\ref{lem:conic-postulation}, removing $3v+5$ copies of the
conic leaves degree $5(u-v-2)$ and multiplicity $2(u-v-2)$.
The residual parameter is nonnegative, and this proves the formula for $H$.

Now set $p=q-5$. The multiplicity for $J_e$ in degree $p$ is
$2u+v-4$, and the conic intersection is $-(3v-10)$.
Removing $3v-10$ copies leaves degree $5(u-v+3)$ and
multiplicity $2(u-v+3)$. Thus the kernel dimension of
$S_{p-e}^{\,5}\to S_p$ is
\begin{align*}
 \dim_k\mathcal Z_p(f)
 &=\frac{5(u-v+3)^2+5(u-v+3)+2}{2}\\
 &\quad-\binom{5u+v-3}{2}+5\binom{2u+v-3}{2}
 =\binom{3v-10}{2}.
\end{align*}
Substituting in~\eqref{eq:inherited-hf-bound} proves the first inequality.

We have $q<2d$. The untruncated coefficient at $q$ is
$\binom{q+2}{2}-5\binom{q-d+2}{2}=G(u,v)$.
It is positive: the expression $2G$ decreases with $v$ in our range and
\[
 2G(u,v)\ge \frac94u^2-\frac{15}{2}u-4>0\qquad(u\ge12).
\]
For $d\le a<2d$, the ratio
$\binom{a+2}{2}/\binom{a-d+2}{2}$ decreases with $a$.
Thus all preceding coefficients are positive as well, and truncation
does not change $G(u,v)$.
Finally,
\[
 G(u,v)+\binom{3v-10}{2}-H(u,v)=10u-40v+50\ge10,
\]
whereas $\binom{3v-10}{2}\ge1$. This proves strict excess over both bounds.
\end{proof}

\begin{proof}[Proof of Theorem~\ref{thm:intro-inheritance}]
Take $v=4$ in Theorem~\ref{thm:ternary-family}. Then
$\binom{3v-10}{2}=1$ and $G(u,4)-H(u,4)=10u-111>0$. For general pairs with $u\ge19$,
Proposition~\ref{prop:exact-lower-ternary} gives equality since
$q=5u+4\le\lfloor(21u+5)/4\rfloor-2$.
\end{proof}

This gives a negative answer to Shapiro's question whether the maximum
of these two bounds equals the generic Hilbert function
\cite[Question~4.1]{Shapiro2026}. For example, $u=12,v=4$ gives $d=37$,
$q=64$ and
\[
 G(12,4)=115,\qquad H(12,4)=106,\qquad \dim_k(S/I)_{64}\ge116.
\]
The excess is not confined to isolated degrees.

\begin{corollary}\label{cor:ternary-interval}
For $u\ge128$ and
$\lceil u/8\rceil\le v\le\lfloor u/6\rfloor$, the ideal of
Theorem~\ref{thm:ternary-family} satisfies
\[
 \dim_k(S/I)_{5u+v}-\max\{G(u,v),H(u,v)\}
 \ge10u-40v+50\ge\frac{10}{3}u+50.
\]
\end{corollary}
\begin{proof}
The stated range is contained in that of
Theorem~\ref{thm:ternary-family}. Moreover,
\[
 2(H-G)=9v^2+17v+10-20u
 \ge\frac{9}{64}u^2-\frac{143}{8}u+10>0.
\]
The last polynomial is positive at $u=128$ and increasing for $u\ge128$.
The result follows from~\eqref{eq:ternary-strict-bound} and its proof.
\end{proof}

\subsection{An excess in the leading term}
We now let the degrees of the additional factors grow. This changes
the scale of the discrepancy: it can have the same quadratic order
as the Hilbert function itself.

Write $N(a)=\binom{a+2}{2}$ for $a\ge0$, and $N(a)=0$ for $a<0$.
Let $H_1,\ldots,H_5$ be nonzero forms of a common degree $s\ge1$,
and put
\[
 I=(L_1^eH_1,\ldots,L_5^eH_5),\qquad d=e+s,\qquad
 h_e(q)=\dim_k(S/J_e)_q.
\]
The field and the general linear forms $L_i$ are as above.

\begin{proposition}\label{prop:growing-factor-bound}
Suppose that $p=q-5s$ satisfies $e\le p\le q\le2e-2$ and
$3q-5e+5-15s\ge2$. Set
\[
 G=N(q)-5N(q-d),\qquad
 B=G+\binom{3q-5e+5-15s}{2}.
\]
Then $\dim_k(S/I)_q\ge B$ and
\begin{equation}\label{eq:growing-factor-difference}
 B-h_e(q)=s(70e-40q+110s-60).
\end{equation}
If $G>0$, it is the coefficient in degree $q$ of the
positive-truncated generic-form series, and
\begin{equation}\label{eq:growing-factor-excess}
 \dim_k(S/I)_q-\max\{G,h_e(q)\}
 \ge \min\left\{
 \binom{3q-5e+5-15s}{2},\
 s(70e-40q+110s-60)\right\}.
\end{equation}
\end{proposition}
\begin{proof}
For $a=p$ or $a=q$, put $t_a=3a-5e+5$ and $w_a=2e-a-2$.
Both are nonnegative. Starting with the fat-point system
$(a;(a-e+1)^5)$, remove $t_a$ copies of the conic.
Its intersection with the conic is negative before each removal,
and the residual system is $(5w_a;(2w_a)^5)$.
Lemma~\ref{lem:conic-postulation} and~\eqref{eq:ternary-apolarity} give
\[
 h_e(a)=\frac{5w_a^2+5w_a+2}{2},\qquad
 \dim_k\mathcal Z_a(L_1^e,\ldots,L_5^e)
 =h_e(a)-N(a)+5N(a-e)=\binom{t_a}{2}.
\]
Lemma~\ref{lem:relation-transfer} transfers the relations in degree
$p$ to degree $q=p+5s$, proving the bound $B$.
Expanding the binomials gives~\eqref{eq:growing-factor-difference}.
Here $d\le q<2d$. If $G>0$, the decreasing-ratio argument in the proof
of Theorem~\ref{thm:ternary-family} shows that all preceding
coefficients are positive. Subtracting the larger comparison value
from $B$ proves~\eqref{eq:growing-factor-excess}.
\end{proof}

\begin{theorem}\label{thm:leading-excess}
Fix real numbers $\sigma,\lambda_1,\lambda_2$ satisfying
\[
 0<\sigma<\frac1{27},\qquad
 \frac53+5\sigma<\lambda_1<\lambda_2<
 \frac74+\frac{11}{4}\sigma.
\]
There is a constant $c>0$, depending only on
$\sigma,\lambda_1,\lambda_2$, with the following property.
For every sufficiently large integer $e$, set
$s=\lfloor\sigma e\rfloor$ and $d=e+s$.
For general $L_1,\ldots,L_5$ and arbitrary nonzero
$H_1,\ldots,H_5\in S_s$, one has
\[
 \dim_k(S/I)_q-\max\{[t^q][(1-t^d)^5/(1-t)^3]_+,h_e(q)\}
 \ge ce^2
\]
whenever $\lceil\lambda_1e\rceil\le q\le\lfloor\lambda_2e\rfloor$.
The predicted coefficient is positive throughout this interval.
\end{theorem}
\begin{proof}
Put $\lambda=q/e$. Uniformly for $\lambda\in[\lambda_1,\lambda_2]$,
the two differences in~\eqref{eq:growing-factor-excess}, divided by
$e^2$, tend respectively to
\[
 \frac{(3\lambda-5-15\sigma)^2}{2},
 \qquad \sigma(70-40\lambda+110\sigma).
\]
Both have a positive minimum on this interval. The hypotheses of
Proposition~\ref{prop:growing-factor-bound} hold for sufficiently
large $e$: in particular,
$p/e\ge\lambda_1-5\sigma>5/3$ and
$q/e\le\lambda_2<50/27<2$.

It remains to check the positive truncation. The untruncated
coefficient divided by $e^2$ tends uniformly to
\[
 \gamma(\lambda,\sigma)
 =\frac{\lambda^2-5(\lambda-1-\sigma)^2}{2}.
\]
This function decreases with $\lambda$ in the interval under
consideration. At $U=7/4+11\sigma/4$ it is
\[
 \gamma(U,\sigma)=\frac{1-14\sigma-31\sigma^2}{8}
 \ge\frac{40}{729}>0\qquad(0\le\sigma\le1/27).
\]
Thus $G>0$ for sufficiently large $e$, uniformly in $q$.
The proposition now proves the assertion.
\end{proof}

For a concrete integer family, take $e=360m$, $s=10m$ and $m\ge5$.
For every $654m\le q\le657m$, the excess over both comparisons is
at least $72m^2$. Indeed, the two terms on the right-hand side
of~\eqref{eq:growing-factor-excess} are bounded below by
$\binom{12m+5}{2}$ and $200m^2-600m$, respectively.
The generic coefficient decreases over this interval, and at its
right endpoint it is $9902m^2-1167m-4>0$.
There are $3m+1$ such degrees, so their total excess is at least
$216m^3$.

The interval allowed by the theorem has width $(1-27\sigma)/12$
in the normalized degree $q/e$. Thus $1/27$ is the boundary for
this particular comparison argument to give a positive-width interval
with quadratic excess over both bounds. It is not asserted to be
a boundary for the actual Hilbert function.

The order of the discrepancy can also be bounded from above. This gives
a uniform statement for fixed, sublinear and proportional factor degrees.

\begin{corollary}\label{cor:sharp-discrepancy}
Fix $0<\sigma_0<1/27$ and $0<\delta<(1-27\sigma_0)/24$.
There are constants $c>0$ and $e_0$, depending only on $\sigma_0,\delta$,
such that the following holds. For integers $e\ge e_0$ and
$1\le s\le\sigma_0e$, let $I=(L_1^eH_1,\ldots,L_5^eH_5)$, where
the $L_i$ are general and the $H_i\in S_s$ are arbitrary nonzero forms.
Put
\[
 \Delta(q)=\dim_k(S/I)_q-
 \max\{[t^q][(1-t^{e+s})^5/(1-t)^3]_+,h_e(q)\}.
\]
For every integer $q$ in the interval
\begin{equation}\label{eq:sharp-discrepancy-interval}
 \left(\frac53+\delta\right)e+5s
 \le q\le
 \left(\frac74-\delta\right)e+\frac{11}{4}s,
\end{equation}
the predicted coefficient is positive and
\[
 cse\le\Delta(q)\le5se.
\]
The interval contains a number of integer degrees proportional to $e$,
uniformly in $s$. The sum of $\Delta(q)$ over these degrees therefore
has order $se^2$, with constants depending only on $\sigma_0,\delta$.
\end{corollary}
\begin{proof}
There is a graded surjection
\[
 \bigoplus_{i=1}^5(S/(H_i))(-e)\longrightarrow J_e/I,
 \qquad ([A_i])_i\longmapsto\sum_i L_i^e A_i.
\]
For $a=q-e\ge s$, it gives
\[
 \Delta(q)\le\dim_k(J_e/I)_q
 \le5\bigl(N(a)-N(a-s)\bigr)
 =\frac{5s}{2}(2a-s+3)\le5s(a+1).
\]
On~\eqref{eq:sharp-discrepancy-interval}, $q\le2e-2$ for all
sufficiently large $e$, uniformly in $s$, proving the upper bound.

For the lower bound, Proposition~\ref{prop:growing-factor-bound}
applies uniformly on this interval. Indeed, $q-5s>e$ and
$3q-5e+5-15s\ge3\delta e+5$.
With $\sigma=s/e$, the positivity calculation in the proof of
Theorem~\ref{thm:leading-excess} is uniform for $0\le\sigma\le\sigma_0$,
so $G>0$ once $e$ is sufficiently large.
The two terms in~\eqref{eq:growing-factor-excess} are at least
\[
 \frac{9\delta^2}{2}e^2
 \ge\frac{9\delta^2}{2\sigma_0}se,
 \qquad s(40\delta e-60)\ge20\delta se,
\]
where the last inequality holds for $e\ge3/\delta$.
Their minimum proves the lower bound. Finally, the length of
\eqref{eq:sharp-discrepancy-interval} is
\[
 \left(\frac1{12}-\frac{9s}{4e}-2\delta\right)e
 \ge\left(\frac{1-27\sigma_0}{12}-2\delta\right)e>0.
\]
Integer rounding changes the count by at most two. Summing the
pointwise bounds proves the last assertion.
\end{proof}

Containment and transfer of relations give different constraints on
the Hilbert function. In the ranges above, the transferred relations
force an excess over both the pure-power and generic-form values.
The quotient $J_e/I$ supplies a complementary upper bound, so the
order of this excess is determined uniformly in the factor degree.
The next subsection identifies the precise rank conditions needed to
turn these bounds into an exact Hilbert function.

\subsection{From lower bounds to an exact Hilbert function}
\label{sec:critical-degrees}
We now compare the lower bound obtained from inherited relations with
the pure-power Hilbert function. Their maximum gives the exact
Hilbert function precisely when a restriction map has maximal rank
in every degree. Its kernel and cokernel propagate in opposite
directions, reducing this question to two consecutive degrees in
the five-generator ternary case.

We first work in a standard graded polynomial ring $S$ in at least two variables over an
infinite field. Let $f_1,\ldots,f_r$ have degree $e$, let the nonzero
forms $H_i$ have degree $s\ge1$, and set
\[
 J=(f_1,\ldots,f_r),\qquad I=(f_1H_1,\ldots,f_rH_r),\qquad
 Q=\prod_{i=1}^r H_i.
\]
Write $\mathcal Z=\mathcal Z(f)$, graded by total relation degree as
above. Componentwise reduction modulo the $H_i$ induces
\[
 \rho:V=\mathcal Z/Q\mathcal Z\longrightarrow
 D=\bigoplus_{i=1}^r(S/(H_i))(-e).
\]
Thus $V_q=\mathcal Z_q/Q\mathcal Z_{q-rs}$. Put $E=\ker\rho$ and
$C=J/I$. There is an exact sequence of graded modules
\begin{equation}\label{eq:restriction-sequence}
 0\longrightarrow E\longrightarrow V\xrightarrow{\rho}D
 \longrightarrow C\longrightarrow0,
\end{equation}
where the last map sends $([A_i])_i$ to $\sum_i f_iA_i$ modulo $I$.
Indeed, if this sum belongs to $I$, subtracting a suitable vector
$(H_iB_i)_i$ from $(A_i)_i$ produces a relation among the $f_i$.
The kernel of reduction on $\mathcal Z$ consists of the vectors
$(H_iB_i)_i$ with $\sum_i f_iH_iB_i=0$. Modding out by
$Q\mathcal Z$ removes exactly the transferred relations of
Lemma~\ref{lem:relation-transfer}.

Set $N_S(q)=\dim_k S_q$, with negative components understood to be
zero, and put
\[
 a_q=N_S(q)-rN_S(q-e-s)+\dim_k\mathcal Z_{q-rs},\qquad
 h_q=\dim_k(S/J)_q.
\]
The multiplication presentation of $I$ and the sequence
$0\to J/I\to S/I\to S/J\to0$ give
\begin{equation}\label{eq:restriction-hf-identities}
 \dim_k(S/I)_q=a_q+\dim_k E_q=h_q+\dim_k C_q.
\end{equation}

\begin{lemma}\label{lem:two-degree-propagation}
If $p\ge e$, $E_p=0$ and $C_{p+1}=0$, then
\[
 \dim_k(S/I)_q=\max\{a_q,h_q\}\qquad(q\ge0).
\]
The same conclusion holds if $E_p=C_p=0$.
\end{lemma}
\begin{proof}
Choose a linear form $\ell$ relatively prime to $Q$.
Multiplication by $\ell$ is injective on $V$: if
$\ell z=Qw$ with $z,w\in\mathcal Z$, unique factorization implies
that every component of $z$ is divisible by $Q$. Write $z=Qu$.
The identity $\sum_i f_i z_i=0$ then implies
$\sum_i f_i u_i=0$, so $u\in\mathcal Z$.
Consequently multiplication by $\ell$ is also injective on $E$,
and $E_p=0$ implies $E_q=0$ for all $q\le p$.

The module $C=J/I$ is generated in degree $e$, by the classes of
the $f_i$. Hence $C_{q+1}=S_1C_q$ for $q\ge e$, and
$C_{p+1}=0$ implies $C_q=0$ for all $q\ge p+1$.
Equation~\eqref{eq:restriction-hf-identities} now gives the value
$a_q$ below the cut and $h_q$ above it. Each is necessarily the
larger bound. If $E_p=C_p=0$, the two values agree at the cut.
\end{proof}

There is also a useful restriction on possible failures. By
\eqref{eq:restriction-hf-identities}, the excess over the larger bound is
\[
 \dim_k(S/I)_q-\max\{a_q,h_q\}
   =\min\{\dim_k E_q,\dim_k C_q\}.
\]
The proof shows that the nonzero degrees of $E$ form an upper
interval, whereas those of $C$ form a lower interval starting at $e$.
Thus the degrees in which $\rho_q$ fails maximal rank have no gaps.

Return to five ternary powers and linear additional factors.
Keep $h_e(q)=\dim_k(S/J_e)_q$ and $N(q)=\dim_k k[x,y,z]_q$.

\begin{proposition}\label{prop:two-critical-degrees}
Let $e\ge36$, let $L_1,\ldots,L_5$ have distinct coefficient points
on a smooth conic, and let $M_1,\ldots,M_5$ be nonzero linear forms
over a field of characteristic zero. For $f_i=L_i^e$, $H_i=M_i$ and
$I=(f_iM_i)_i$, put
\[
 q_- =\left\lfloor\frac{7e+5}{4}\right\rfloor,\qquad
 q_+ =\left\lceil\frac{7e+5}{4}\right\rceil.
\]
Then the formula
\begin{equation}\label{eq:full-hf-candidate}
 \dim_k(S/I)_q=
 \max\{N(q)-5N(q-e-1)+\dim_k\mathcal Z_{q-5},h_e(q)\}
 \quad(q\ge0)
\end{equation}
holds if and only if $\rho_{q_-}$ is injective and
$\rho_{q_+}$ is surjective. If $e\equiv1\pmod4$, these are the
same degree and the condition is invertibility of one square map.
\end{proposition}
\begin{proof}
For $q=q_-$ or $q=q_+$, both $q$ and $q-5$ lie in the conic
range of Lemma~\ref{lem:conic-postulation}.
Writing $\tau=3q-5e+5$ and $w=2e-q-2$, that calculation gives
\[
 h_e(q)=\frac{5w^2+5w+2}{2},\qquad
 \dim_k\mathcal Z_q=\binom{\tau}{2},\qquad
 \dim_k\mathcal Z_{q-5}=\binom{\tau-15}{2}.
\]
Here $\tau-15\ge0$, $w\ge0$ and $q-5\ge e$ for $e\ge36$.
In particular,
\[
 \dim_k V_q=15\tau-120,\qquad \dim_k D_q=5(q-e+1),
 \qquad a_q-h_e(q)=70e-40q+50.
\]
The last difference is nonnegative at $q_-$ and nonpositive at
$q_+$. Equality in~\eqref{eq:full-hf-candidate} therefore forces
$E_{q_-}=0$ and $C_{q_+}=0$. Conversely these vanishings imply
the formula by Lemma~\ref{lem:two-degree-propagation}, since the
two degrees are consecutive or equal.
\end{proof}

For general factors, the proposition reduces the exact Hilbert-function
question to two actual restriction maps whose dimensions grow linearly
with $e$. For each fixed $e$, a nonzero minor for each required rank
would suffice: the two nonempty open conditions can be intersected.
Proving those ranks uniformly in $e$ remains open. In particular,
the reduction does not assume that these structured maps behave as
matrices with independent entries.

\section{A converse for relations on a conic}
\label{sec:polyharmonic}

The transfer in Lemma~\ref{lem:relation-transfer} produces relations
whose components share every additional factor. We now give a converse
for coefficient points on a conic, subject to a bound on the number of
applications of its associated differential operator. The conclusion
concerns the entire relation, including cancellation between its
different differential layers.

Throughout this section, $S=k[x,y,z]$, where $k$ has characteristic zero,
and
\[
 L(T)=x+Ty+T^2z,\qquad
 \mathcal D=\partial_y^2-\partial_x\partial_z.
\]
The coefficient points of $L(T)$ parametrize a smooth conic, and
$\mathcal D L(T)^a=0$. This is the case $g=\eta^2-\xi\zeta$ of the
apolarity identity $g(\partial)\ell^a=a(a-1)\cdots(a-i+1)\,g(c)\,\ell^{a-i}$
for a form $g$ of degree $i$ and a linear form $\ell$ with coefficient
vector $c$; see~\cite{IarrobinoKanev1999}. Thus
$\ker\mathcal D^{K+1}$ plays the role of polyharmonic polynomials for
the conic. We give the local basis and the interpolation argument needed
here explicitly.

\begin{theorem}\label{thm:polyharmonic-common-factors}
Let $r\ge1$, let $t_1,\ldots,t_r\in k$ be distinct, and choose
$s_i\in k$ with $s_i\ne t_i$ for every $i$. Set
$L_i=L(t_i)$ and $M_i=L(s_i)$. Let $K,e,m$ be integers with
$K\ge0$, $e\ge K+1$, $m\ge K$, and put $q=e+m$. Suppose that
\begin{equation}\label{eq:polyharmonic-range}
 (r-1)(m+1)<2q-4K+2.
\end{equation}
If $F_1,\ldots,F_r\in S_q$ satisfy
\[
 L_i^eM_i\mid F_i,\qquad \mathcal D^{K+1}F_i=0,
 \qquad \sum_{i=1}^r F_i=0,
\]
then every $F_i$ is divisible by
\[
 Q=\prod_{s\in\{s_1,\ldots,s_r\}}L(s),
\]
where each distinct secondary parameter occurs once in the product.
If there are at least $2K+1$ distinct secondary parameters, then
$F_1=\cdots=F_r=0$.
\end{theorem}

The secondary parameters may repeat or coincide with other primary
parameters. When the two sets are disjoint, division by $Q$ expresses
the relation as a transferred relation among the pure powers $L_i^e$.
The differential condition is essential to the argument.

\subsection{Local differential kernels}

For a polynomial in an auxiliary variable $h$, write $[h^a]$ for
coefficient extraction. We use $(b)_j=b(b-1)\cdots(b-j+1)$, with
$(b)_0=1$.

\begin{lemma}\label{lem:polyharmonic-local-basis}
Under the numerical assumptions of
Theorem~\ref{thm:polyharmonic-common-factors}, a basis of
$L_i^eS_m\cap\ker\mathcal D^{K+1}$ is
\begin{equation}\label{eq:polyharmonic-basis}
 z^j[h^a]L(t_i+h)^{q-j},
 \qquad 0\le j\le K,\quad 0\le a\le m-j.
\end{equation}
\end{lemma}
\begin{proof}
Every displayed polynomial is divisible by $L_i^e$, and
\[
 \mathcal D\bigl(z^j L(T)^{q-j}\bigr)
 =-j(q-j)z^{j-1}L(T)^{q-j-1}.
\]
It is therefore killed by $\mathcal D^{j+1}$.
Applying $\mathcal D^K$ to a linear dependence leaves only the
$j=K$ terms, as nonzero multiples of
$[h^a]L(t_i+h)^{q-2K}$. These jets are independent: in coordinates
$U=L_i$, $V=y+2t_i z$, $W=z$, their terms of $W$-degree zero are
the distinct monomials $\binom{q-2K}{a}U^{q-2K-a}V^a$.
Here $a\le m-K\le q-2K$. Descending in $j$ proves independence.

In the same coordinates, $\mathcal D=\partial_V^2-\partial_U\partial_W$.
For $A\ge1$, $b\ge1$, the map
\[
 \mathcal D:U^A S_b\longrightarrow U^{A-1}S_{b-1}
\]
is onto. Indeed, integrating a target monomial
$U^{A-1+p}V^vW^w$ by $-\partial_U\partial_W$ introduces only
an error with $V$-degree $v-2$ from $\partial_V^2$; successive
corrections terminate. All integration denominators are nonzero
positive integers. Iteration gives
\[
 \dim\bigl(L_i^eS_m\cap\ker\mathcal D^{K+1}\bigr)
 =\dim S_m-\dim S_{m-K-1}
 =\sum_{j=0}^K(m-j+1).
\]
For $m=K$, the target of the last derivative is zero and the same
formula holds. This is the number of vectors in
\eqref{eq:polyharmonic-basis}.
\end{proof}

\subsection{From local divisibility to a common factor}

The proof has two parts. The relation bounds the degrees of numerators
in a rational interpolation representation. Divisibility by $M_i$
then forces one of these numerators to have a zero of larger order
than its degree, so that it vanishes identically.

\begin{proof}[Proof of Theorem~\ref{thm:polyharmonic-common-factors}]
By Lemma~\ref{lem:polyharmonic-local-basis}, there are unique proper
rational functions $f_j(T)$, $0\le j\le K$, with poles only at the
$t_i$, of order at most $m-j+1$, such that
\begin{equation}\label{eq:polyharmonic-residues}
 F_i=\sum_{j=0}^K z^j\operatorname{Res}_{T=t_i}
                  f_j(T)L(T)^{q-j}.
\end{equation}
Here a residue means the coefficient of $(T-t_i)^{-1}$ in a rational
function; equivalently, it is the residue of that function times $dT$.
Proper means that the function vanishes at infinity. Its principal
parts encode precisely the coefficients of the local basis.

Put $\mu_j(l)=\sum_i\operatorname{Res}_{T=t_i}f_j(T)T^l$.
For $a+b+c=q$, the coefficient of $x^ay^bz^c$ in $\sum_iF_i$,
multiplied by $a!b!c!$, is
\begin{equation}\label{eq:polyharmonic-moments}
 \sum_{j=0}^{\min(K,c)}(q-j)!(c)_j\mu_j(b+2c-2j)=0.
\end{equation}
Fix $w=b+2c$. The possible $c$ run from $\max(0,w-q)$ to
$\lfloor w/2\rfloor$. For $w\le q$, the left side is a polynomial
in $c$ of degree at most $\min(K,\lfloor w/2\rfloor)$, with more
zeros than its degree. For $q<w\le2q-2K$, there are at least
$K+1$ possible values of $c$, and the degree is at most $K$.
The falling factorials are a basis, so in both cases their coefficients
vanish. Consequently
\[
 \mu_j(l)=0\qquad(0\le l\le2q-2K-2j).
\]
These moments are the coefficients of $T^{-l-1}$ in the expansion
of $f_j$ at infinity. With $P(T)=\prod_i(T-t_i)$, we obtain
\begin{equation}\label{eq:polyharmonic-numerator-bound}
 f_j=\frac{B_j}{P^{m-j+1}},\qquad
 \deg B_j\le r(m+1)-2q+2K-2-(r-2)j.
\end{equation}
A negative bound means that $B_j=0$.

Restrict to $L(s)=0$ by setting $(x,y,z)=(-su,u-s,1)$.
Then $L(T)=(T-s)(u+T)$. Integration by parts in the residue gives
\[
 F_i|_{L(s)=0}=\operatorname{Res}_{T=t_i}g_s(T)(u+T)^q,
 \qquad
 g_s=\sum_{j=0}^K\frac{(-1)^j}{(q)_j}
       \frac{d^j}{dT^j}\bigl(f_j(T)(T-s)^{q-j}\bigr).
\]
At $s=s_i$, this restriction vanishes. The pole order of $g_{s_i}$
at $t_i$ is at most $m+1$, while the first $m+1$ jets of
$(u+T)^q$ at $t_i$ are linearly independent as polynomials in $u$.
Thus $g_{s_i}$ has no principal part at $t_i$.

Factoring out $(T-s)^{q-2K}$ and clearing denominators gives the
polynomial
\begin{equation}\label{eq:polyharmonic-polynomial}
 R_s(T)=P^{m+1}\sum_{j=0}^K\frac{(-1)^j}{(q)_j}
 \sum_{a=0}^j\binom ja(q-j)_{j-a}
 f_j^{(a)}(T)(T-s)^{2K-2j+a}.
\end{equation}
It satisfies $g_s=(T-s)^{q-2K}R_s/P^{m+1}$.
The degree in $T$ of each summand is at most
\[
 E=r(m+1)-2q+4K-2.
\]
Indeed, after differentiating $a$ times and clearing to denominator
$P^{m+1}$, the numerator has degree at most
$\deg B_j+rj-a$; now use~\eqref{eq:polyharmonic-numerator-bound}.
Since $s_i\ne t_i$, regularity of $g_{s_i}$ at $t_i$ implies
$(T-t_i)^{m+1}\mid R_{s_i}$. Condition~\eqref{eq:polyharmonic-range}
says exactly that $E<m+1$. Hence $R_{s_i}=0$ identically.

For any secondary parameter $s$, it follows that $g_s=0$.
The residue formula for the restriction now gives
$F_l|_{L(s)=0}=0$ for every $l$. These distinct linear factors
are relatively prime, so their product $Q$ divides every $F_l$.

Finally, $R_s$ has degree at most $2K$ in $s$. If it vanishes at
at least $2K+1$ distinct secondary parameters, it is identically
zero as a polynomial in $s$ and $T$. Its coefficient of $s^{2K}$
gives $f_0=0$. Inductively, after $f_0,\ldots,f_{j-1}$ vanish,
the coefficient of $s^{2K-j}$ is a nonzero scalar times
$P^{m+1}f_j^{(j)}$. A proper rational function whose $j$th
derivative is zero is zero in characteristic zero. Thus all $f_j$,
and therefore all $F_i$, vanish.
\end{proof}

\begin{corollary}\label{cor:polyharmonic-inheritance}
Suppose the primary and secondary parameter sets in
Theorem~\ref{thm:polyharmonic-common-factors} are disjoint, and let
$v$ be the number of distinct secondary parameters. Every relation
covered by that theorem has the form
\[
 F_i=QP_i,\qquad P_i\in L_i^eS_{q-v-e},\qquad \sum_iP_i=0.
\]
In particular, for five pairs with all ten parameters distinct,
this conclusion holds whenever $e\ge K+1$ and
$e+K\le q\le2e-2K-2$.
\end{corollary}
\begin{proof}
Divide by $Q$, which is relatively prime to every $L_i$.
For $r=5$, inequality~\eqref{eq:polyharmonic-range} is equivalent,
for integral $q$, to $q\le2e-2K-2$.
\end{proof}

At the two critical degrees of
Proposition~\ref{prop:two-critical-degrees}, the corollary applies
for $e\ge36$ and
$0\le K\le\lfloor(2e-q-2)/2\rfloor$.
It excludes additional relations annihilated by these powers of
$\mathcal D$, uniformly in $e$. Relations of larger differential
order remain possible; the maximal-rank question in
Proposition~\ref{prop:two-critical-degrees} therefore remains open.

\section{A uniform bound for the full Hilbert function}
\label{sec:uniform-hilbert}

For five ternary generators with linear additional factors, the
inherited relations give more than a lower bound: together with the
pure-power comparison, they determine every component of the generic
Hilbert function up to an absolute additive constant. Throughout this
section, $S=k[x,y,z]$ and $\operatorname{char}k=0$. We use the total-degree
grading on the relation module $\mathcal Z$ and the modules $E,C$ of
\eqref{eq:restriction-sequence}, with $f_i=L_i^e$ and $H_i=M_i$.
Put $J_e=(L_i^e)_{i=1}^5$, $I_e=(L_i^eM_i)_{i=1}^5$, and
\begin{equation}\label{eq:uniform-comparison}
 B_e(q)=\max\{a_q,h_e(q)\},\qquad
 a_q=N(q)-5N(q-e-1)+\dim_k\mathcal Z_{q-5},
 \qquad h_e(q)=\dim_k(S/J_e)_q,
\end{equation}
where $N(q)=\dim_k S_q$ and negative graded components are zero.

\begin{theorem}\label{thm:uniform-hilbert-bound}
For every integer $e\ge36$ and five general pairs of linear forms
$(L_i,M_i)$, one has, simultaneously for all integers $q$,
\[
 0\le \dim_k(S/I_e)_q-B_e(q)\le55.
\]
Moreover $(S/I_e)_q=0$ for $q\ge2e+54$.
The nonempty open set of pairs may depend on $e$.
\end{theorem}

The bound is independent of both the exponent and the graded degree,
and Proposition~\ref{prop:explicit-comparison} below writes $B_e$ as
an explicit piecewise polynomial. The proof uses no machine
computation. It controls the kernel and cokernel of the restriction
map in complementary ranges. We first place all ten coefficient
points on a conic. Relations then admit short rational descriptions
on each secondary line. Compatibility at the pairwise intersections
imposes forty independent conditions. These conditions give the
required cokernel bound, while multiplication by a linear form
propagates the kernel bound to lower degrees. Above degree $2e-2$,
Macaulay's growth bound and Gotzmann's persistence theorem finish
the argument.

\subsection{Differential order of relations}

Retain $L(T)=x+Ty+T^2z$ and
$\mathcal D=\partial_y^2-\partial_x\partial_z$ from
Section~\ref{sec:polyharmonic}.

\begin{lemma}\label{lem:all-relations-differential-order}
Let $t_1,\ldots,t_5$ be distinct, $L_i=L(t_i)$, and
$e\le q\le2e-2$. Write $\tau=3q-5e+5$.
If $\tau\le1$, there are no relations among the $L_i^e$ in
total degree $q$. If $\tau\ge2$, every such relation, written in
polynomial components as $\sum_i F_i=0$ with $L_i^e\mid F_i$,
satisfies $\mathcal D^{\tau-1}F_i=0$ for every $i$.
\end{lemma}
\begin{proof}
The operator $\mathcal D$ maps $L_i^eS_{q-e}$ onto
$L_i^{e-1}S_{q-e-1}$ by the triangular monomial integration in
Lemma~\ref{lem:polyharmonic-local-basis}. Its kernel has basis
$[h^a]L(t_i+h)^q$, $0\le a\le q-e$. The space of harmonic
degree-$q$ forms has basis $[T^a]L(T)^q$, $0\le a\le2q$:
these polynomials have disjoint monomial weights, and the same
triangular integration shows that the dimension is $2q+1$.
In these bases the sum map from the five component kernels has
matrix entries $\binom aj t_i^{a-j}$. Its transpose takes Taylor
jets through order $q-e$ at the five points. Hermite interpolation
therefore gives rank $\min\{2q+1,5(q-e+1)\}$ and kernel dimension
$\max\{\tau-1,0\}$.

Applying $\mathcal D$ componentwise maps relations for $(e,q)$ to
relations for $(e-1,q-2)$. The kernel is precisely this harmonic
relation space. If $\tau\le1$, this map is injective. Each
iteration decreases $\tau$ and the coefficient degree $q-e$ by
one. After $q-e+1$ iterations the coefficient degree is negative;
all intervening exponents stay positive because $q\le2e-2$.
This proves vanishing when $\tau\le1$. If $\tau\ge2$, apply
$\mathcal D$ a total of $\tau-1$ times. The resulting relation
has parameter $1$ and is zero. This many iterations are allowed
since $\tau-1\le q-e$ in the stated range.
\end{proof}

\subsection{The comparison function}
The comparison value $B_e(q)$ has four regimes. The inherited
correction appears at $q\approx5e/3$, the two comparison values cross
at $q=(7e+5)/4$, and both vanish from $2e-1$ on.

\begin{proposition}\label{prop:explicit-comparison}
Let $e\ge36$, and let $L_1,\ldots,L_5$ have distinct coefficient points
on a smooth conic. For $q\ge0$ put $w=2e-q-2$. Then
\[
 B_e(q)=
 \begin{cases}
 N(q)-5N(q-e-1), & q<\dfrac{5e+12}{3},\\[1.2ex]
 N(q)-5N(q-e-1)+\dbinom{3q-5e-10}{2},
   & \dfrac{5e+12}{3}\le q\le\dfrac{7e+5}{4},\\[1.2ex]
 \dfrac{5w^2+5w+2}{2}, & \dfrac{7e+5}{4}\le q\le2e-2,\\[1.2ex]
 0, & q\ge2e-1.
 \end{cases}
\]
In the first two ranges the maximum in~\eqref{eq:uniform-comparison}
is $a_q$; in the last two it is $h_e(q)$.
\end{proposition}
\begin{proof}
Write $t_a=3a-5e+5$. For $e\le a\le2e-2$ we have
$\dim_k\mathcal Z_a=\binom{t_a}{2}$ when $t_a\ge2$, by the conic
calculation in the proof of Proposition~\ref{prop:growing-factor-bound},
and $\mathcal Z_a=0$ when $t_a\le1$, by
Lemma~\ref{lem:all-relations-differential-order}. Since $\mathcal Z_a$
is the kernel of $S_{a-e}^{\,5}\to(J_e)_a$, for every $a$
\begin{equation}\label{eq:pure-power-relation-count}
 h_e(a)=N(a)-5N(a-e)+\dim_k\mathcal Z_a .
\end{equation}
Put $\tau=t_q$.

If $q<e$, then $a_q=h_e(q)=N(q)$. Suppose $e\le q<(5e+12)/3$,
so that $\tau\le16$ and $q\le2e-2$. Then $\mathcal Z_{q-5}=0$: either
$q-5<e$, or $t_{q-5}=\tau-15\le1$. Hence
\[
 a_q-h_e(q)=5(q-e+1)-\dim_k\mathcal Z_q .
\]
This is nonnegative: either $\mathcal Z_q=0$, or $2\le\tau\le16$ and
$5(q-e+1)=\tfrac53(\tau+2e-2)\ge120\ge\binom\tau2$.

Suppose $(5e+12)/3\le q\le2e-2$. Then $\tau\ge17$ and $q-5\ge e$,
so both relation spaces are given by the binomial formula and
\[
 a_q-h_e(q)=5(q-e+1)-\binom\tau2+\binom{\tau-15}2=70e-40q+50.
\]
This is nonnegative exactly when $q\le(7e+5)/4$. Moreover
$(5e+12)/3\le(7e+5)/4$ because $e\ge33$. The value of $h_e(q)$ in
the third range is the conic calculation already cited.

Finally let $q\ge2e-1$. By~\eqref{eq:ternary-apolarity},
$h_e(q)=h_Z(q;m)$ with $m=q-e+1$, and $q-2m=2e-2-q<0$. The
quantity $a-2m$ does not change under $(a,m)\mapsto(a-2,m-1)$, and
$a<2m$ implies $2a<5m$ when $m\ge1$. Lemma~\ref{lem:conic-postulation}
therefore lets us remove conics until $a\le1$, where $h_Z(a;m)=0$
because five points of a conic are not collinear. Thus $h_e(q)=0$.
For $a_q$, put $p=q-5$. If $p\ge2e-1$, the same argument and
\eqref{eq:pure-power-relation-count} give
$\dim_k\mathcal Z_p=5N(p-e)-N(p)$. If $2e-6\le p\le2e-2$, the conic
calculation adds $\varepsilon=(5i^2+5i+2)/2\le51$, where
$i=2e-2-p\le4$. Using $N(n)-N(n-5)=5n-5$ and $N(n)-N(n-4)=4n-2$
for $n\ge3$, we obtain
\[
 a_q=20e-15q+25+\varepsilon\le-10e+91<0 .
\]
Hence $B_e(q)=0$.
\end{proof}

For $q\le(7e+5)/4$ the term $G_e(q)=N(q)-5N(q-e-1)$ is the coefficient
of $t^q$ in the generic-form series
$\trunc{(1-t^{e+1})^5/(1-t)^3}$. Indeed, at $q=(7e+5)/4$ one has
$32\,G_e(q)=4e^2-56e-108>0$, and the decreasing-ratio argument in the
proof of Theorem~\ref{thm:ternary-family} shows that all earlier
coefficients are positive. Thus the first regime is the generic-form
prediction, the second adds the inherited correction
$\binom{3q-5e-10}{2}$ to it, and the third is the pure-power value.

\subsection{Boundary numerators}

For the rest of the proof, take ten distinct parameters $t_i,s_i$
and set $L_i=L(t_i)$, $M_i=L(s_i)$. Fix
\begin{equation}\label{eq:uniform-interval}
 q_0=\left\lceil\frac{7e+2}{4}\right\rceil,
 \qquad q_0\le q\le2e-2,
\end{equation}
and define
\[
 m=q-e,\quad \tau=3q-5e+5,\quad K=\tau-2,
 \quad \lambda=4q-7e+5,\quad b=2\lambda-2,
 \quad j_0=q-2K.
\]
For $e\ge36$ these parameters satisfy
\begin{equation}\label{eq:uniform-range-check}
 \tau\ge16,\quad \lambda\ge7,\quad
 m-K=4e-2q-3\ge1,\quad e-K=6e-3q-3\ge3,
 \quad j_0=10e-5q-6\ge4.
\end{equation}
In particular the local basis of
Lemma~\ref{lem:polyharmonic-local-basis} applies to every relation.

\begin{lemma}\label{lem:uniform-boundary-injection}
In the range~\eqref{eq:uniform-interval}, $E_q$ injects into five
spaces of polynomials $B_j(T)$ of degree at most $b$. For the
polynomials in its image, put
\[
 f^{(j)}(T)=\frac{B_j(T)}{\prod_{i\ne j}(T-t_i)^{m+1}},
 \qquad
 r_{ij}(u)=\operatorname{Res}_{T=t_i}
       f^{(j)}(T)(T-s_j)^{j_0}(u+T)^q\quad(i\ne j),
\]
and $r_{jj}=0$. They satisfy
\begin{align}
 \sum_i r_{ij}(u)&=0,\label{eq:boundary-sum}\\
 r_{ij}(-s_i)&=0 &&(i\ne j),\label{eq:boundary-own}\\
 r_{ij}(-s_l)&=r_{il}(-s_j)
       &&(j<l,\ i\notin\{j,l\}).\label{eq:boundary-match}
\end{align}
\end{lemma}
\begin{proof}
Represent a mixed relation by polynomial components $F_i$ of
degree $q$, with $L_i^eM_i\mid F_i$ and $\sum_i F_i=0$.
Lemma~\ref{lem:all-relations-differential-order} gives $\mathcal D^{K+1}F_i=0$.
Use the residue representation~\eqref{eq:polyharmonic-residues}
and its functions $f_j$. Restriction to $L(s)=0$ gives
\[
 F_i(-su,u-s,1)=\operatorname{Res}_{T=t_i}g_s(T)(u+T)^q,
 \quad
 g_s=\sum_{j=0}^K\frac{(-1)^j}{(q)_j}
       \frac{d^j}{dT^j}\bigl(f_j(T)(T-s)^{q-j}\bigr).
\]
By~\eqref{eq:polyharmonic-numerator-bound},
$f_j(T)=O(T^{-2q+2K+2j-2})$. Each term of $g_s$ is thus
$O(T^{-q+2K-2})$, so $g_s$ is proper by
\eqref{eq:uniform-range-check}. Since $\sum_iF_i=0$, its moments
$\sum_i\operatorname{Res}_{t_i}g_s(T)T^a$ vanish for
$0\le a\le q$. Expanding this proper function at infinity gives
the stronger decay
\begin{equation}\label{eq:boundary-decay}
 g_s(T)=O(T^{-q-2}).
\end{equation}
The restriction formula also shows that $g_s$ has poles of order
at most $m+1$ at the $t_i$ and a zero of order at least $j_0$
at $T=s$. At $s=s_j$, divisibility by $M_j$ removes the pole
at $t_j$, by independence of the local jets of $(u+T)^q$.
Consequently
\[
 B_j(T)=\frac{\prod_{i\ne j}(T-t_i)^{m+1}}{(T-s_j)^{j_0}}
          g_{s_j}(T)
\]
is a polynomial. Its degree is at most
$4(m+1)-q-2-j_0=b$ by~\eqref{eq:boundary-decay}.

These polynomials vanish exactly when every $F_i$ restricts to
zero on every secondary line. For the reverse implication, use
properness and the independent local jets to recover $g_{s_j}=0$
from the zero restrictions. Thus the kernel consists precisely of
relations divisible componentwise by $Q=\prod_jM_j$.
Since $Q$ is coprime to all $L_i$, division by $Q$ identifies
this kernel with the inherited relations. This proves the injection.
The functions $r_{ij}$ are the restrictions just constructed.
Equation~\eqref{eq:boundary-sum} follows by restriction of the
relation. At the intersection of the secondary lines $i,j$ the
$i$th component vanishes, giving~\eqref{eq:boundary-own}.
The parametrizations of lines $j,l$ give the same point
$(s_js_l,-s_j-s_l,1)$, so equality of the two values of $F_i$
gives~\eqref{eq:boundary-match}.
\end{proof}

\begin{proposition}\label{prop:exact-lower-ternary}
Let $e\ge36$. For five general pairs of ternary linear forms and
every $q\le q_- -2$, where $q_-=\lfloor(7e+5)/4\rfloor$, one has
\[
 \dim_k(S/I_e)_q=a_q=B_e(q).
\]
Thus in these grades all relations among the products are inherited
from relations among the pure powers.
\end{proposition}
\begin{proof}
Put $q_*=q_- -2$. We first work with distinct primary parameters
$t_i$ on the conic. At $q=q_*$ the definitions above give
\[
 7e-6\le4q_*\le7e-3,\qquad b=8q_*-14e+8\le2 .
\]
Although $q_*<q_0$, the boundary injection of
Lemma~\ref{lem:uniform-boundary-injection} still applies here:
$\tau\ge e/4+1/2>2$, $m=q_*-e\ge(3e-6)/4$, and $q_*\le2e-2$
gives $m-K\ge1$, $e-K\ge3$ and $j_0\ge4$ by the same identities.
These inequalities suffice for the local jet basis, decay at infinity
and division by $(T-s_j)^{j_0}$ in its proof. The lower bound $b+1\ge13$
used for four-point jet interpolation below is unnecessary here.

If $b<0$, the injection already gives $E_{q_*}=0$. Otherwise fix
an index $j$ and a secondary parameter $s_j$ distinct from the
primary parameters. For each $i\ne j$ the map $B_j\mapsto r_{ij}(u)$
is injective on polynomials of degree at most $b$. Indeed, if
$r_{ij}=0$, independence of the first $m+1$ jets of $(u+T)^q$ at
$t_i$ forces the principal part of $f^{(j)}(T)(T-s_j)^{j_0}$ to
vanish there. This forces $(T-t_i)^{m+1}\mid B_j$, while $b<m+1$.

The evaluations $r_{ij}(-s_i)=0$ for $i\ne j$ now force $B_j=0$
for general choices of the $s_i$. To see this, start with the
space of possible $B_j$, of dimension $b+1\le3$. As long as its
current subspace is nonzero, take a nonzero member and an unused
index $i\ne j$. Its polynomial $r_{ij}(u)$ is nonzero, so a general
choice of $s_i$ cuts the dimension by one. For fixed $j$ the
polynomial $r_{ij}$ depends on the primary parameters and $s_j$, hence
the other four evaluation points can indeed be chosen independently.
This full-rank condition is open. Intersecting the five nonempty
open conditions, one for each $j$, gives $E_{q_*}=0$ on the
paired-conic specialization.

Choose general primary points first; they lie on a smooth conic.
At fixed primary points $a_{q_*}$ is constant as the secondary
factors vary. Since $\dim E_{q_*}=\dim(S/I_e)_{q_*}-a_{q_*}$,
upper semicontinuity gives the same vanishing for general pairs
(the primary relation dimension is constant on a dense open set).
Multiplication by a linear form coprime to $\prod_i M_i$ is injective
on $E$ by Lemma~\ref{lem:two-degree-propagation}. Hence $E_q=0$ for
all $q\le q_*$. Equation~\eqref{eq:restriction-hf-identities} gives
the asserted formula.
\end{proof}

\subsection{Compatibility from binary interpolation}

\begin{lemma}\label{lem:forty-boundary-conditions}
For fixed distinct primary parameters and general secondary
parameters, the twenty equations~\eqref{eq:boundary-own} and
thirty equations~\eqref{eq:boundary-match} have rank forty on
the space of five polynomials of degree at most $b$.
\end{lemma}
\begin{proof}
We first check that the values appearing in these equations can
be prescribed subject only to~\eqref{eq:boundary-sum}.
Fix $j$, and write $a=e+j_0$, $D=q+j_0$. The residues
\[
 R_i(u)=\operatorname{Res}_{T=t_i}f^{(j)}(T)(u+T)^D
 \qquad(i\ne j)
\]
are a relation among the four binary powers $(u+t_iv)^a$ in
degree $D$, after homogenization. Indeed
$b=4(m+1)-D-2$, so $f^{(j)}=O(T^{-D-2})$ and their sum is
zero. Each residue is divisible by $(u+t_i)^a$.
The map from $B_j$ is injective by independence of the local jets
and properness. It is onto the binary relation space because that
space has dimension
\begin{equation}\label{eq:binary-four-dimension}
 \max\{0,4(D-a+1)-(D+1)\}=b+1.
\end{equation}
For completeness, binary apolarity identifies the annihilator of
the sum of the four power spaces with forms vanishing to order
$D-a+1$ at four distinct points of $\mathbb P^1$. Their ideal is
principal, of degree $4(D-a+1)$, which gives this dimension formula.

Let $p=-s_j$, away from the four primary roots. Taking the jets
of the four components through order three maps this relation
space onto the twelve-dimensional space of jets whose sum is
zero. To see surjectivity, its kernel consists of relations
divisible componentwise by $(u-pv)^4$. Division gives the relation
space in degree $D-4$. Formula~\eqref{eq:binary-four-dimension}
shows that its dimension is $(b+1)-12$, since $b+1\ge13$
and $D-a=m\ge28$.

The actual restrictions $r_{ij}$ are obtained by directional
polarization. Direct differentiation of their residue formulas
at $p=-s_j$ gives
\[
 \frac{r_{ij}^{(h)}(p)}{(q)_h}
   =\frac{R_i^{(h)}(p)}{(D)_h},\qquad 0\le h\le3.
\]
The factors are nonzero and independent of $i$, so these jets
also fill the twelve-dimensional zero-sum space. Coalescing the
four evaluation points $-s_l$, $l\ne j$, to $p$ now proves that
the evaluation map is generically onto: divide a nonzero
twelve-evaluation minor by the cube of the four-point Vandermonde
and take its confluent limit. The limit is the corresponding
nonzero jet minor, up to nonzero factorial factors. Hence the
minor is not the zero polynomial. This condition is nonempty
open in the secondary parameters. Intersecting the five such
conditions and excluding collisions proves simultaneous
surjectivity for all $j$.

The total evaluation space has dimension $5\cdot12=60$.
For each ordered pair of distinct secondary lines, its value is
a four-component vector of sum zero, with the component indexed
by the source line already absent. The own-line condition kills
one further coordinate. These twenty conditions are independent.
At an unordered pair $\{j,l\}$, the two remaining vectors each
have three components of sum zero. Matching them imposes two
conditions. The ten unordered pairs concern disjoint value
coordinates, so impose twenty more conditions. Surjectivity of
the evaluation maps proves the asserted rank forty.
\end{proof}

\subsection{Proof of the uniform bound}

\begin{proof}[Proof of Theorem~\ref{thm:uniform-hilbert-bound}]
Equation~\eqref{eq:restriction-hf-identities} gives, for every tuple,
\[
 \dim_k(S/I_e)_q-B_e(q)=\min\{\dim_k E_q,\dim_k C_q\}\ge0.
\]
For the paired-conic tuples of the preceding lemmas, in the range
\eqref{eq:uniform-interval}, the injection and forty compatibility
conditions give
\begin{equation}\label{eq:uniform-E-bound}
 \dim_k E_q\le5(b+1)-40=10\lambda-45.
\end{equation}
Here $q-5\ge e$ and $\tau-15\ge1$. The conic calculation of
Lemma~\ref{lem:conic-postulation} gives
$\dim\mathcal Z_q=\binom\tau2$ and
$\dim\mathcal Z_{q-5}=\binom{\tau-15}2$. Consequently
\begin{align*}
 a_q-h_e(q)
 &=5(q-e+1)-\left\{\binom\tau2-\binom{\tau-15}2\right\}\\
 &=70e-40q+50=100-10\lambda.
\end{align*}
Together with~\eqref{eq:uniform-E-bound}, this yields
$\dim_k C_q=a_q-h_e(q)+\dim_k E_q\le55$.

At $q=q_0$, one has $7\le\lambda\le10$, so
$\dim_k E_{q_0}\le55$. Multiplication by a linear form coprime
to $Q$ is injective on $E$, as proved in
Lemma~\ref{lem:two-degree-propagation}. Thus
$\dim_k E_q\le55$ for every $q\le q_0$. The minimum of the
kernel and cokernel dimensions is therefore at most $55$ in
every degree through $2e-2$.

It remains to pass from the paired-conic specialization to
general pairs. For fixed $e$, only finitely many nonnegative
degrees through $2e-2$ occur, and their nonempty open conditions
on the secondary parameters can be intersected. Five general
primary coefficient points lie on a smooth conic; after scalar
extension and a linear change of coordinates they have the
parametrization used above. The primary Hilbert function and
relation dimensions are constant on this open set. The resulting
tuples are witnesses in the full parameter space of pairs.
Upper semicontinuity of each quotient dimension transfers the
upper bound to a nonempty open set there, in all degrees
$q\le2e-2$. Shrink this set so that no three of the ten lines
$L_i=0$, $M_i=0$ are concurrent. Then $S/I_e$ has finite length:
a common zero of the generators would lie on one line from each of
the five pairs.

Write $h_q=\dim_k(S/I_e)_q$. By
Proposition~\ref{prop:explicit-comparison}, $B_e(2e-2)=1$, so
$h_{2e-2}\le56\le2e-2$. Let $t\ge2e-2$ with $0<h_t\le t$.
The $t$th Macaulay representation of $h_t$ consists of $h_t$
binomial coefficients equal to $1$, so Macaulay's theorem gives
$h_{t+1}\le h_t$; see~\cite{Macaulay1927} and
\cite[Section~4.2]{BrunsHerzog1998}. If equality held, Gotzmann's
persistence theorem~\cite{Gotzmann1978}, \cite[Section~4.3]{BrunsHerzog1998}
would apply, since $I_e$ is generated in degree $e+1\le t$. It would
force $h_{t'}=h_t$ for every $t'\ge t$, contradicting finite length.
Thus $h$ decreases strictly from degree $2e-2$ until it vanishes, and
\[
 h_q\le\max\{0,\,2e+54-q\}\qquad(q\ge2e-2).
\]
In particular $h_q\le55$ for $q\ge2e-1$, and $(S/I_e)_{2e+54}=0$. Since $B_e(q)=0$
for $q\ge2e-1$ by Proposition~\ref{prop:explicit-comparison}, this
proves the assertion simultaneously in all degrees.
Scalar extension does not change the dimensions, so the result
holds over any characteristic-zero field.
\end{proof}

The theorem measures the defect of Shapiro's proposed formula
\cite[Question~4.1]{Shapiro2026} up to the same constant.

\begin{corollary}\label{cor:shapiro-defect}
Let $e\ge36$ and let the pairs be as in
Theorem~\ref{thm:uniform-hilbert-bound}. Put
\[
 P_e(q)=\max\Bigl\{[t^q]\trunc{\frac{(1-t^{e+1})^5}{(1-t)^3}},\,
 h_e(q)\Bigr\},
\]
and let $\delta_e(q)=\min\{\binom{3q-5e-10}{2},\,70e-40q+50\}$ for
$(5e+12)/3\le q\le(7e+5)/4$ and $\delta_e(q)=0$ otherwise. Then
\[
 0\le\dim_k(S/I_e)_q-P_e(q)-\delta_e(q)\le55
 \qquad\text{for every integer }q.
\]
Moreover $\max_q\delta_e(q)=\frac{10}{3}e+O(\sqrt e)$ as
$e\to\infty$.
\end{corollary}
\begin{proof}
For $q\le2e+1$ write $G_e(q)=N(q)-5N(q-e-1)$, so that
$a_q=G_e(q)+\dim_k\mathcal Z_{q-5}\ge G_e(q)$. By the remark following
Proposition~\ref{prop:explicit-comparison}, the truncated coefficient
equals $G_e(q)$ for $q\le(7e+5)/4$. In general it is either $G_e(q)$
or $0$, and it is $0$ for $q\ge2e-1$ because
$G_e(2e-1)=-e(e-7)/2<0$.

In the first regime of Proposition~\ref{prop:explicit-comparison},
$B_e(q)=a_q=G_e(q)\ge h_e(q)$, so $P_e(q)=B_e(q)$. In the second,
$P_e(q)=\max\{G_e(q),h_e(q)\}$ and
$B_e(q)-P_e(q)=\min\{a_q-G_e(q),a_q-h_e(q)\}=\delta_e(q)$ by the
proof of that proposition. In the third, the truncated coefficient
is at most $\max\{a_q,0\}\le h_e(q)$, so $P_e(q)=h_e(q)=B_e(q)$. In
the tail all three quantities vanish. In every case
$P_e(q)+\delta_e(q)=B_e(q)$, and the first assertion is
Theorem~\ref{thm:uniform-hilbert-bound}.

For the maximum, $70e-40q+50\le(10e-330)/3$ on the window. Conversely,
let $q_1=\lceil(5e+12)/3\rceil$ and $q=q_1+j$ with
$j=\lceil\sqrt{20e/27}\,\rceil$, which lies in the window for large
$e$. Then $\binom{3q-5e-10}{2}\ge\binom{3j+2}{2}\ge\frac{10}{3}e$ and
$70e-40q+50\ge(10e-410)/3-40j$.
\end{proof}

Thus Shapiro's formula is correct up to $55$ outside a window of
relative width $1/12$, and inside it the inherited relations account
for a defect that grows linearly with $e$.

The theorem gives $\dim_k(S/I_e)_q=B_e(q)+O(1)$ uniformly in
$e$ and $q$. In particular, in any degree regime where $B_e(q)$
has a quadratic and linear expansion in $e$, the same two terms
describe the actual Hilbert function. Floors and residue classes
in $q$ must be retained in computing the linear term. Summing the
pointwise bounds gives a total-length error at most
$55(2e-1)+\binom{56}{2}$. The bound is slightly better below the
crossing: for $q\le q_0$ the discrepancy is at most
$10\lambda(q_0)-45$, that is, $25,35,45$ or $55$ according as
$e\equiv2,3,0$ or $1\pmod4$.

The low-degree equality also identifies where the largest discrepancy in
Shapiro's comparison is attained. Write $Q=(7e+5)/4$. By
Proposition~\ref{prop:exact-lower-ternary}, the discrepancy
$\dim_k(S/I_e)_q-P_e(q)$ equals $\delta_e(q)$ for
$q\le\lfloor Q\rfloor-2$. At any larger integer $q$ in the window one has
$\delta_e(q)\le40(Q-q)<80$, while outside the window $\delta_e(q)=0$.
By Corollary~\ref{cor:shapiro-defect}, the discrepancy in these remaining
degrees is at most $135$. Since $\max_q\delta_e(q)$ grows linearly with $e$,
for all sufficiently large $e$ the largest discrepancy is exactly
$\max_q\delta_e(q)$ and occurs in the range of exact equality.

The terminal calculation gives a further refinement.
Theorem~\ref{thm:terminal}, which is computer-assisted, shows that
$(S/I_e)_q=0$ already for $q\ge2e-1$, so the error in the tail is
zero. The exact equality with $B_e(q)$ in all degrees is
equivalent to the two critical ranks of
Proposition~\ref{prop:two-critical-degrees}, and remains open.

\section{The last nonzero component}\label{sec:terminal}

The excess established in Section~\ref{sec:inherited} does not persist
to the last nonzero component of the pure-power quotient. We determine
that component for linear additional factors. The proof is partly
computer-assisted: a deformation reduces the assertion for all large
exponents to polynomial identities in one parameter. We give the reduction
and degree bounds before stating the exact arithmetic calculation.

\begin{theorem}\label{thm:terminal}
Let $k$ have characteristic zero and let $S=k[x,y,z]$. For every integer
$e\ge12$, five general pairs of linear forms $(L_i,M_i)$ satisfy
\[
 \dim_k(S/(L_1^eM_1,\ldots,L_5^eM_5))_{2e-2}=1,
 \qquad (S/(L_1^eM_1,\ldots,L_5^eM_5))_q=0\quad(q\ge2e-1).
\]
In particular, the regularity of this quotient is $2e-2$.
The nonempty open set of pairs is allowed to depend on $e$.
\end{theorem}

Put $T=k[X,Y,Z]$, with the usual differentiation pairing between $S$
and $T$. For a linear form $L$, write $D_L=L(\partial_X,\partial_Y,
\partial_Z)$, and let $P_L$ be its coefficient point in $\mathbb P^2$.
We use the inverse-system description of powers of linear forms
\cite{EmsalemIarrobino1995}. If $\mathfrak p_i$ is the ideal of $P_{L_i}$,
then the inverse space of $J=(L_1^e,\ldots,L_5^e)$ in degree $q\ge e$ is
$(\bigcap_i\mathfrak p_i^{q-e+1})_q$.

For five points on a smooth conic $C$, B\'ezout's theorem gives
\begin{equation}\label{eq:terminal-pure}
 (S/J)_{2e-2}^{\vee}=kC^{e-1},\qquad (S/J)_{2e-1}=0.
\end{equation}
Indeed, five multiplicities $e-1$ in degree $2e-2$ force $e-1$
successive conic factors. Five multiplicities $e$ in degree $2e-1$
force more conic factors than the degree allows. Since $I=(L_i^eM_i)
\subseteq J$, the first equality gives a lower bound of one for
$\dim_k(S/I)_{2e-2}$. This conic lower bound also occurs in
\cite[Corollary~3.3]{Shapiro2026}. The issue here is equality.
Once equality holds, $(J/I)_{2e-2}=0$. The module $J/I$ is generated in
degree $e$, so it vanishes in all subsequent degrees. Together with
\eqref{eq:terminal-pure}, this proves the vanishing assertion in
Theorem~\ref{thm:terminal}. It remains to determine one component.

\subsection{A deformation with fixed residual spaces}

Work first over $\mathbb Q$. Fix
\begin{equation}\label{eq:terminal-points}
 C=XZ-Y^2,\qquad P_i=(1,i,i^2),\qquad L_i=x+iy+i^2z
 \quad(0\le i\le4),
\end{equation}
and choose the coefficient vectors of $N_i$ to be, in this order,
\begin{equation}\label{eq:terminal-directions}
 (1,2,3),\ (2,5,7),\ (3,11,13),\ (5,17,19),\ (7,23,29).
\end{equation}
Consider $I(t)=(L_i^e(L_i+tN_i))$ and put $q=2e-2$, $m=e-2$.
A lift $F(t)=\sum_{j=0}^K t^jF_j$ in its degree-$q$ inverse space,
modulo $t^{K+1}$, satisfies exactly
\begin{equation}\label{eq:terminal-lift}
 D_{L_i}F_j+D_{N_i}F_{j-1}\in\mathfrak p_i^m,
 \qquad 0\le j\le K,\quad F_{-1}=0.
\end{equation}
This follows by first applying $D_{L_i+tN_i}$ and then using apolarity
for $L_i^e$ in degree $q-1$.

Define spaces independent of $e$ by
\[
 V_j=\left(\bigcap_{i=0}^4\mathfrak p_i^{4j+4}\right)_{10j+10}.
\]
Appendix~\ref{sec:terminal-basis} constructs an integral basis and gives
$\dim V_j=10j^2+25j+16$. For local equations, $\mathfrak m_i$ denotes
the maximal ideal at $P_i$ in the chart $X=1$; set
\[
 Q_j=\bigoplus_{i=0}^4
 \mathfrak m_i^{4j+5}/\mathfrak m_i^{5j+5},
 \qquad \dim Q_j=\frac{5j(9j+11)}2.
\]

\begin{lemma}\label{lem:terminal-reduction}
Suppose $e\ge5K+6$. Every solution of \eqref{eq:terminal-lift} has
a unique expression $F_j=C^{e-6-5j}G_j$, with $G_j\in V_j$.
Here $G_0$ is arbitrary in $V_0$, and the remaining equations are
\begin{equation}\label{eq:terminal-blocks}
 B_j(e)G_j+A_j(e)G_{j-1}=0\quad\text{in }Q_j\qquad(1\le j\le K),
\end{equation}
where the component at $P_i$ is obtained from
\begin{align*}
 B_j(E)G&=(E-6-5j)(D_{L_i}C)G+C D_{L_i}G,\\
 A_j(E)H&=C^5\bigl((E-1-5j)(D_{N_i}C)H+C D_{N_i}H\bigr).
\end{align*}
Conversely, these equations give every truncated lift.
\end{lemma}

\begin{proof}
At $P_i$, the radial operator $D_{L_i}$ on a degree-$q$ homogeneous
form acts on a local jet of order $r$ by the nonzero scalar $q-r$,
up to terms of higher order. Induction in \eqref{eq:terminal-lift}
therefore gives $\operatorname{ord}_{P_i}F_j\ge m-j$:
$D_{N_i}$ lowers the preceding multiplicity by at most one.
The intersection excess with the conic is
$5(m-j)-2q=e-6-5j$. Removing a conic factor lowers this excess by one,
so B\'ezout forces precisely the asserted divisibility. The residual
degree and required multiplicity are $10j+10$ and $4j+4$.

Writing $a_j=e-6-5j$, differentiation gives the identity
\[
 C(D_{L_i}F_j+D_{N_i}F_{j-1})
   =C^{a_j}\bigl(B_j(e)G_j+A_j(e)G_{j-1}\bigr).
\]
Multiplication by $C$ raises local order by one at each $P_i$.
Thus \eqref{eq:terminal-lift} is equivalent to order at least
$m+1-a_j=5j+5$ for the bracket. The bracket already has order
$4j+5$, giving exactly $Q_j$. This also proves the equivalence when
$a_j=0$. At $j=0$, the radial operator preserves order $m$, so every
$G_0\in V_0$ is allowed.
\end{proof}

\subsection{Elimination without positive parameter poles}

We describe a left inverse of $B_j(E)$ on its image. At $P_i$ use
$u=Y-i$ and $c=Z-Y^2$ on $X=1$. These are regular local coordinates.
On a degree-$d$ form the radial operator becomes
$d-u\partial_u+(u^2-c)\partial_c$. Consequently
\begin{equation}\label{eq:terminal-local-B}
 B_j(E)g=c(E+5j+4-u\partial_u-c\partial_c)g
            +u^2(E-6-5j+c\partial_c)g.
\end{equation}
For $h=B_j(E)g$, recover the coefficients of $g$ in increasing total
degree $r=a+b$, for $4j+4\le r\le5j+3$, by
\begin{equation}\label{eq:terminal-inverse}
 g_{a,b}=\frac{h_{a,b+1}-(E-6-5j+b+1)g_{a-2,b+1}}
                      {E+5j+4-r}.
\end{equation}
Negative indices mean zero. The second coefficient has total degree
$r-1$, so the recurrence is triangular. On arbitrary $h$, use the
same recurrence and retain the unused equations as compatibility
conditions. The denominators are among $E+1,\ldots,E+j$.
If $D_j(E)=\prod_{h=1}^j(E+h)$, the resulting local inverse multiplied
by $D_j$ is polynomial of degree at most $j-1$.

The constant jet map
\[
 J_j:V_j\longrightarrow\bigoplus_i
       \mathfrak m_i^{4j+4}/\mathfrak m_i^{5j+4}
\]
has kernel $C^{5j}V_0$: the stronger multiplicities have conic excess
$5j$, and division reduces to degree ten and multiplicity four.
Choose a complement to this kernel, invert $J_j$ on that complement,
and extend the inverse linearly to the jet target. Composing with
\eqref{eq:terminal-inverse} gives $R_j(E):Q_j\to V_j$ such that
\[
 R_jB_j=\text{projection onto the complement},\qquad B_jR_jB_j=B_j.
\]
All global splitting denominators are nonzero rational constants;
the only parameter poles are negative integers. The ancillary code
specifies the complements by pivot columns and verifies the inverses
over $\mathbb Q$. No surjectivity of $J_j$ onto its whole target is used.

Starting with $G_0=H_0$, equations \eqref{eq:terminal-blocks} become
\begin{align}
 G_j&=-R_jA_jG_{j-1}+C^{5j}H_j,\qquad H_j\in V_0,
       \label{eq:terminal-recursion}\\
 (1-B_jR_j)A_jG_{j-1}&=0.\label{eq:terminal-compatibility}
\end{align}
The line $kC^5\subset V_0$ corresponds to the persistent inverse
$C^{e-1}$. It has zero coupling to every subsequent compatibility
equation: $A_j(C^{5j})$ has local order at least $5j+5$.
Remove this line from each $H_j$ and choose a complement $U$ of
dimension fifteen. The order-$K$ compatibility matrix $\mathcal M_K(E)$
then has $15K$ columns, from $H_0,\ldots,H_{K-1}$; $H_K$ does not enter.
Write $\mathcal M_K=[\mathcal M_{\rm initial}\mid\mathcal M_{\rm later}]$,
with fifteen initial columns. Linear algebra gives
\begin{equation}\label{eq:terminal-lift-dimension}
 \dim\{H_0\in U:\ H_0\text{ lifts through order }K\}
   =15-\rank\mathcal M_K+\rank\mathcal M_{\rm later}.
\end{equation}

These matrices admit a small, proved interpolation bound. Put
$D^{(j)}=D_1\cdots D_j$. Inductively, multiplying the coefficients
of $G_j$ in \eqref{eq:terminal-recursion} by $D^{(j)}$ gives
polynomials of degree at most $j(j+1)/2$. Multiplying the order-$j$
compatibility equation by the same product gives degree at most
$j(j+1)/2+1$, since $A_j$ and $B_j$ are affine in $E$ and
$D_jR_j$ has degree at most $j-1$. Clearing the constant denominators
does not change this bound. In particular, for $K=6$ the polynomial
matrix has 2625 rows, 90 columns, and entry degrees at most 22.

\subsection{Exact arithmetic and the proof of the theorem}

\begin{lemma}[Exact rank calculation]\label{lem:terminal-certificate}
For the points and directions in \eqref{eq:terminal-points}--
\eqref{eq:terminal-directions}, the denominator-cleared order-six
matrix satisfies
\[
 \rank\mathcal M_6(e)=41,\qquad
 \rank\mathcal M_{\rm later}(e)=26\qquad(e\ge36).
\]
For each integer $12\le e\le35$, the degree-$(2e-2)$ multiplication
matrix of $(L_i^e(L_i+N_i))$ has rank $\binom{2e}{2}-1$ over $\mathbb Q$.
\end{lemma}

\begin{proof}[Computer-assisted verification]
The directory \texttt{anc/terminal} contains the complete reconstruction
and checking code. It uses exact integer and rational arithmetic through
\texttt{python-flint} version~0.9.0. We describe separately the
identities that give an upper bound and the nonvanishing certificates
that give a lower bound.

The code constructs the integral bases of Appendix~\ref{sec:terminal-basis}
and the constant jet splittings. It evaluates the cleared matrix at
$E=0,\ldots,22$ and reconstructs its entries by polynomial interpolation.
The degree bound above makes this reconstruction exact. An additional
evaluation at $E=36$ agrees in every entry; this is a check of the
implementation, not the source of uniformity. There are 146 nonzero
rows after normalization by their polynomial gcds. Every removed factor
is recorded.

Exact elimination over $\mathbb Q(E)$ gives ranks 41 and 26. To control
specialization, the code checks 49 polynomial kernel vectors of the
75-column later matrix. Their distinct nonzero free coordinates show
independence over $\mathbb Q(E)$. Hence every 27-minor is identically
zero, and the later rank is at most 26 at every specialization.
This argument does not require the kernel vectors to remain independent
after specialization.

A second elimination in reverse row and column order records all row
divisors and pivots needed for 41 pivots of the full matrix. Among these
are 135 distinct nonconstant factors. For every such factor $p$, all
nonzero coefficients of $p(36+u)$ have the same sign and the constant
coefficient is nonzero. Thus every operation is valid for real $E\ge36$,
and the full rank is at least 41 there. It is at most the later rank
plus fifteen, so both asserted ranks follow. The checker repeats the
kernel identities and sign tests independently of the first elimination.

Finally, for $12\le e\le35$ the code constructs the raw multiplication
matrix over $\mathbb F_{32003}$, with $\binom{2e}{2}$ rows and
$5\binom{e-1}{2}$ columns. Multinomial expansion and repeated linear
multiplication give identical generators. The factorial-weighted
pairing with $C^{e-1}$ annihilates every column, and exact modular
elimination gives rank $\binom{2e}{2}-1$ in each case. A nonzero minor
modulo 32003 is a nonzero integer minor. The conic upper bound on rank
from \eqref{eq:terminal-pure} proves equality over $\mathbb Q$.
\end{proof}

\begin{proof}[Proof of Theorem~\ref{thm:terminal}]
For $e\ge36$, take $K=6$ in Lemma~\ref{lem:terminal-reduction}.
Lemmas~\ref{lem:terminal-reduction} and~\ref{lem:terminal-certificate},
together with \eqref{eq:terminal-lift-dimension}, show that no
nonpersistent initial class lifts through order six.
If the inverse space of $I(t)$ over $\mathbb Q((t))$ had dimension
greater than one, take a second vector, clear its negative powers of
$t$, subtract its coefficients along the persistent line, and divide
by its first remaining power of $t$. Its initial coefficient would be
nonpersistent and would lift through every order, a contradiction.
The inverse dimension over $\mathbb Q(t)$ is therefore one. A nonzero
minor then supplies a rational specialization of $t$ with the same
dimension, for each $e$. The finite part of
Lemma~\ref{lem:terminal-certificate} handles $12\le e\le35$.

Upper semicontinuity of the quotient dimension, combined with the
universal conic lower bound on the open set of five points on a smooth
conic, gives equality for general pairs. The rational minors remain
nonzero over every characteristic-zero field. The propagation argument
following \eqref{eq:terminal-pure} proves vanishing above degree $2e-2$.
For an Artinian standard graded quotient, regularity is the last nonzero
degree, proving the final assertion.
\end{proof}

The theorem identifies the terminal inverse class and the degree at which
$J/I$ has vanished. It does not determine the full Hilbert function in
the intervening degrees. In particular, terminal equality with the
pure-power quotient is consistent with the excess over both comparison
bounds proved in Section~\ref{sec:inherited}.

For $e=13$ the theorem agrees with Shapiro's first five-generator
failure~\cite[Theorem~1.3]{Shapiro2026}: there the Hilbert series is
$G_{3,14,5}(t)+t^{24}$, so the quotient has dimension one in degree
$24=2e-2$ and vanishes from degree $25$ on. Theorem~\ref{thm:terminal}
shows that this terminal shape persists for every $e\ge12$.
Without any computation, Theorem~\ref{thm:uniform-hilbert-bound}
gives the weaker statement $\dim_k(S/I)_q\le55$ for $q\ge2e-1$ and
vanishing for $q\ge2e+54$ when $e\ge36$.

\section{A non-Lefschetz scheme in four variables}\label{sec:non-lefschetz}

The one-dimensional inverse space in Theorem~\ref{thm:terminal} also
determines local equations for a four-variable multiplication map.
This gives a geometric consequence of the terminal calculation: the
additional linear factors produce both isolated points and embedded
curves in a degree-specific non-Lefschetz scheme.

In this section $k$ is algebraically closed of characteristic zero,
$S=k[x,y,z,w]$, and
\[
 A=S/(L_1^eM_1,\ldots,L_6^eM_6),\qquad q=2e-2.
\]
For general pairs and $e\ge12$, the source of multiplication
$A_{q-1}\to A_q$ has larger dimension than the target, as verified
below. Let $\mathcal N_q\subset\mathbb P(S_1)$ be the subscheme
defined by the maximal minors of the universal map
\[
 A_{q-1}\otimes\mathcal O_{\mathbb P(S_1)}(-1)
   \longrightarrow A_q\otimes\mathcal O_{\mathbb P(S_1)}.
\]
Thus $\mathcal N_q$ records failure of surjectivity in this degree,
with its determinantal scheme structure; see
\cite{BoijMiglioreMiroRoigNagel2018}. We use no comparison between
non-Lefschetz schemes in different degrees. For ideals generated by
powers of linear forms, failure of the weak Lefschetz property is
governed by fat points through inverse systems, and is closely tied to
unexpected curves and hypersurfaces
\cite{SchenckSeceleanu2010,HarbourneSchenckSeceleanu2011,
MMN2012,DiGennaroIlardiValles2014,
CookHarbourneMiglioreNagel2018,HarbourneMiglioreNagelTeitler2021};
see \cite{MiglioreNagel2013} for a survey. Here the additional factors
$M_i$ make the scheme structure visible.

The six coefficient points $[L_i]$ determine a classical Weddle
quartic $W\subset\mathbb P(S_1)$~\cite{Weddle1850}, the locus of
vertices of quadric cones through the six points; see~\cite{Dolgachev2020}
for its history. Its general point $[\ell]$ is a
projection center for which the images of the six points in
$\mathbb P(S_1/k\ell)$ lie on a conic. We use the classical reduced,
integral quartic for general points. For Weddle schemes of projected point configurations see~\cite{Weddle2022}. Its determinantal description
and relation to powers of linear forms are discussed in
\cite[Sections~4.1 and~6, Appendix~A]{Weddle2026}.

\begin{theorem}\label{thm:non-lefschetz}
For each integer $e\ge12$, six general pairs $(L_i,M_i)$ have the
following properties. The scheme $\mathcal N_q$ has $W$ as a
generically reduced surface component and has the six distinct
points $[M_i]$ as reduced isolated components outside $W$.
It has at least six distinct embedded associated curves contained
in $W$. At a general point of each of these curves its local ideal
is $(a^2,ab)$, where $a,b$ are part of a regular system of parameters
and $a=0$ defines $W$.
Moreover, multiplication by $M_i$ has cokernel dimension one and
kernel dimension $(e-1)(e-7)$ from degree $q-1$ to degree $q$.
The nonempty open set of pairs may depend on $e$.
\end{theorem}

The curves in this statement are embedded scheme structure inside
the surface, rather than additional components of its underlying
topological space. We do not classify the entire scheme. Nonreduced
and mixed-dimensional Weddle schemes already occur in
\cite{Weddle2026}; here the assertion concerns this mixed-product
family and the local ideal above for every $e\ge12$.

\subsection{Reduction to one inverse polynomial}

Specializing $M_i=L_i$ and using the first branch of
\cite[Proposition~7.3]{BoijLundqvist2026} shows that the six
degree-$(e+1)$ generators have no relations in degrees at most $2e$
on a nonempty open set. Consequently, for $j=q-1,q$,
\begin{equation}\label{eq:lefschetz-dimensions}
 \dim A_j=\binom{j+3}{3}-6\binom{j-e+2}{3},\qquad
 \dim A_{q-1}-\dim A_q=e^2-8e+6>0.
\end{equation}

Fix an index $i$. On an affine chart of projection centers write
$R=S/(\ell)=k[x,y,z]$, and put $l_j=\overline L_j$,
$m_j=\overline M_j$. Restrict to an open set $U_i$ where the five
points $[l_j]$, $j\ne i$, are distinct and determine a smooth conic
$C$ with an invertible interpolation minor, where their five
mixed generators have degree-$q$ quotient dimension one, and where
$l_i\ne0$. Normalize one conic coefficient so that $C$ varies
regularly on the chart. Theorem~\ref{thm:terminal} supplies the
required rank condition for general five pairs. Put
\[
 l=l_i,\quad m=m_i,\quad a=C(l),\quad B=D_lC,\quad b=D_mB.
\]
Here $a,b$ are functions of the center; the differentiation operators
act only on the dual variables $X,Y,Z$, not on these parameters.

\begin{lemma}\label{lem:lefschetz-presentation}
On $U_i$, the ideal of $\mathcal N_q$ is the coefficient ideal of
the dual polynomial
\[
 F=D_mD_l^e C^{e-1}.
\]
The local equation of $W$ is $a=0$, up to a unit.
\end{lemma}

\begin{proof}
The cokernel of multiplication by $\ell$ in degree $q$ is
$(R/(l_j^em_j))_q$. Its zeroth Fitting ideal can therefore be
computed either from the multiplication map or from the ternary
Macaulay matrix. The other five generators have a free rank-one
quotient locally. Indeed, an invertible minor gives rank at least
$\dim R_q-1$, while $C^{e-1}$ annihilates every column: for
$C(l_j)=0$, the polynomial $C(X+t l_j)^{e-1}$ has degree at most
$e-1$ in $t$. Thus $D_{l_j}^eC^{e-1}=0$.
Eliminating this invertible block reduces the presentation to one
row, consisting of the pairings of $l^em$ times the degree-$(e-3)$
monomials with $C^{e-1}$. These are the coefficients of $F$, up to
invertible factorial scalings and a common unit. Finally, eliminating
the five independent rows in the six-point conic interpolation
matrix leaves $C(l)$, giving the claimed local Weddle equation.
\end{proof}

\subsection{The local equations}

Taylor expansion gives an identity valid with parameter coefficients:
\begin{equation}\label{eq:lefschetz-contraction}
 D_l^eC^{e-1}
 =e!\sum_{h=1}^{\lfloor e/2\rfloor}
 \frac{(e-1)!}{h!(e-2h)!(h-1)!}
 a^h B^{e-2h}C^{h-1}.
\end{equation}
Indeed, $C(X+tl)=C+tB+t^2a$, and the summand counts the ways to
obtain the coefficient of $t^e$. Write $D_l^eC^{e-1}=aT$, and set
$c_1=e!(e-1)$ and $c_2=e!(e-1)!/(2(e-4)!)$. Then
\begin{equation}\label{eq:lefschetz-first-terms}
 D_mT=c_1(e-2)b B^{e-3}
 +c_2a\bigl((D_mC)B^{e-4}+(e-4)bCB^{e-5}\bigr)
 \pmod{(a^2)}.
\end{equation}
In particular, all its coefficients belong to $(a,b)$.

\begin{lemma}\label{lem:lefschetz-local}
On $U_i$, the local ideal of $\mathcal N_q$ is as follows:
\begin{enumerate}
\item if $a\ne0$, it is generated by the three coefficients of $m$;
\item if $a=0$ and $b\ne0$, it is $(a)$;
\item if $a=b=0$ and $l,m$ are linearly independent, it is
$(a^2,ab)$.
\end{enumerate}
\end{lemma}

\begin{proof}
If $C(l)\ne0$, the three first derivatives of
$P=D_l^eC^{e-1}$ are independent. To see this, change coordinates
and rescale to put $C=X^2+Y^2+Z^2$ and $D_l=\partial_Z$.
The polynomial $P$ contains nonzero terms $Z^{e-2}$,
$X^2Z^{e-4}$ and $Y^2Z^{e-4}$, and is even in $X,Y$.
Its three derivatives are nonzero and have distinct parity in
$(X,Y)$, so they are independent. An invertible coefficient minor
persists locally. Since $F=D_mP$ is linear in the three coefficients
of $m$, its coefficient ideal is exactly their ideal.

If $a=0$, $b\ne0$, the polynomial $D_mT$ has a unit coefficient
by \eqref{eq:lefschetz-first-terms}, since $B\ne0$. As $F=aD_mT$,
the ideal is $(a)$. In the remaining case, let $\mathfrak n$ be
the maximal ideal of the local ring and let $K$ be the coefficient
ideal of $D_mT$. Modulo $\mathfrak n(a,b)$ its coefficients are
those of
\[
 c_1(e-2)bB^{e-3}+c_2a(D_mC)B^{e-4}.
\]
The two displayed polynomials multiplying $b$ and $a$ are independent:
the nondegenerate conic identifies the independent vectors $l,m$
with independent linear forms $B,D_mC$. Hence
$(a,b)=K+\mathfrak n(a,b)$, and Nakayama's lemma gives $K=(a,b)$.
Multiplication by $a$ proves the third assertion.
\end{proof}

At a point where $a,b$ are regular parameters,
\begin{equation}\label{eq:lefschetz-embedded}
 (a^2,ab)=(a)\cap(a^2,b).
\end{equation}
The prime $(a,b)$ is embedded. Equivalently, the class of $a$ in
the quotient has annihilator $(a,b)$. The next calculation will
ensure that the curves obtained by choosing different omitted
indices are distinct.

\begin{lemma}\label{lem:conic-tangency}
Let $n=5$ or $6$, and let $[l_1],\ldots,[l_n]$ be distinct points
on a smooth conic $C$. Put $b_j=D_{m_j}D_{l_j}C$. For $e\ge3$,
the degree-$(2e-2)$ inverse system of $(l_j^em_j)$ contains
\[
 C^{e-2}\{Q\in k[X,Y,Z]_2:b_jQ(l_j)=0\text{ for all }j\}.
\]
If precisely $t$ of the $b_j$ vanish, this subspace has dimension
$1+t$ for $n=5$ and $\max(1,t)$ for $n=6$.
\end{lemma}

\begin{proof}
For a quadratic $Q$, extract the coefficient of $s^e$ in
$(C+sD_{l_j}C)^{e-2}
 (Q+sD_{l_j}Q+s^2Q(l_j))$ and differentiate in $m_j$. This gives
\[
 D_{m_j}D_{l_j}^e(C^{e-2}Q)
   =e!(e-2)b_jQ(l_j)(D_{l_j}C)^{e-3}.
\]
Restriction of quadratics to $C\simeq\mathbb P^1$ is the complete
space of binary quartics, with kernel $kC$. Any at most five
distinct points impose independent conditions on this space;
vanishing at six distinct points forces the restriction to be zero.
Multiplication by $C^{e-2}$ is injective, giving the stated dimensions.
\end{proof}

These dimensions are lower bounds for the full inverse system.
In particular, a five-generator quotient of dimension one on a
smooth conic requires all five $b_j$ to be nonzero.

\subsection{Existence for general pairs}

\begin{proof}[Proof of Theorem~\ref{thm:non-lefschetz}]
First consider $[\ell]=[M_i]$. Restriction of the other ten forms
to $S/(M_i)$ is surjective on coefficient spaces. The terminal
rank-one condition therefore holds on a nonempty open set of
original pairs. Independently, $\overline L_i$ can be chosen off
the conic through the other five primary points. Intersecting these
six open sets with the early-rank open set used in
\eqref{eq:lefschetz-dimensions} gives simultaneous generality.
Lemma~\ref{lem:lefschetz-local}(1) now identifies the local ideal
with the maximal ideal at $[M_i]$: the coefficients of
$\overline M_i$ are local coordinates there. The point is reduced
and isolated, with cokernel dimension one. Equation
\eqref{eq:lefschetz-dimensions} gives kernel dimension
$e^2-8e+7=(e-1)(e-7)$. The six points are distinct and outside $W$
on a further nonempty open set.

We next prove that the tangency loci meet the required rank charts
for general original pairs, for every fixed $e$. Index the following
example by $j=0,\ldots,5$. Take
\[
 L_j=x+jy+j^2z+(j^3+2j+1)w,\qquad M_5=y+10z+3w,
 \qquad \ell=w.
\]
The projected conic is $C=XZ-Y^2$. The last two projected vectors
are $l_5=(1,5,25)$ and $m_5=(0,1,10)$, so they are independent
and $a=b=0$. For $\ell=w+ux+vy+sz$, normalize the conic through
the first five projected points by setting its $XZ$ coefficient
equal to one. Differentiating the five interpolation equations
and then $a=C(l_5)$ and $b=D_{m_5}D_{l_5}C$ at the origin gives
\begin{equation}\label{eq:lefschetz-jacobian}
 da=(-120,0,0),\qquad db=(1576,-740,74).
\end{equation}
To see these derivatives, write $A(T)=T^3+2T+1$,
$h(T)=T^2\,du-2T\,dv+ds$, and $W_0(T)=\prod_{j=0}^4(T-j)$.
The variation of the normalized conic, restricted to $(1,T,T^2)$,
is the quartic interpolating $A(j)h(j)$ at the five nodes. It is
$A(T)h(T)-W_0(T)\,du$. Evaluating at $5$ gives
$da=-W_0(5)\,du$. Since $m_5=(0,1,10)$ is the tangent vector of
this parametrization at $5$ and its fourth coordinate before projection
is $3$, differentiation gives
$db=(A'(5)-3)h(5)-W_0'(5)\,du$. Now $W_0(5)=120$,
$W_0'(5)=274$, and $A'(5)-3=74$, giving the two independent rows.
The calculation involves only the
primary vectors and $M_5$. The first five residual vectors remain
free. Their projected primary vectors are exactly those of
\eqref{eq:terminal-points}; Lemma~\ref{lem:terminal-certificate}
and its deformation argument therefore provide residual vectors
with terminal dimension one for every fixed $e\ge12$.

Let $\mathcal P$ be the irreducible parameter space of the twelve
nonzero four-variable forms. In a conic interpolation chart over
$\mathcal P\times\mathbb P(S_1)$, impose $a_i=b_i=0$.
The relative Jacobian with respect to the three center coordinates
has rank two at the example above. The resulting incidence projects
smoothly to $\mathcal P$, with relative dimension one, near that
point. Restrict further to the open conditions defining $U_i$,
independence of $l_i,m_i$, and rank two of the relative Jacobian.
The restriction is nonempty for each $e$, by the fixed-primary
terminal argument. A smooth morphism is open, so its image contains
a nonempty open subset of $\mathcal P$. Permuting the six indices
and intersecting the resulting opens shows that general pairs have
such a curve for every omitted index. By
Lemma~\ref{lem:lefschetz-local}(3) and
\eqref{eq:lefschetz-embedded}, the closure of an irreducible
component of each good tangency locus is an embedded associated
curve of $\mathcal N_q$.

On the smooth-conic locus of $W$, each condition $b_j=0$ is
intrinsic: changing the normalization of $C$ multiplies $b_j$ by
a unit. At a good point of the curve indexed by $i$, the other
five products have terminal dimension one. Lemma~\ref{lem:conic-tangency}
implies $b_j\ne0$ for $j\ne i$. A curve indexed by $j$ is contained
in $b_j=0$, so it cannot have the same irreducible closure. This
proves the existence of six distinct embedded associated curves.

Finally, the inverse class $C^{e-1}$ shows that a dense open of $W$
lies in $\mathcal N_q$, hence all of $W$ does. At a transverse
tangency point, $a,b$ are independent local parameters. Within
$a=0$ there are nearby points with $b\ne0$, still in $U_i$.
Lemma~\ref{lem:lefschetz-local}(2) makes the scheme reduced there.
Since the classical Weddle quartic is integral for general primary
points, it is a generically reduced surface component, as claimed.
\end{proof}

On the union of the $U_i$, the support is exactly the intersection
of that union with $W\cup\{[M_1],\ldots,[M_6]\}$. The complement
includes degenerate conic projections and five-generator rank drops,
which the argument does not classify. The result determines local
geometry of one multiplication map; the full generic Hilbert
function of the six-generator quotient remains open.

\section{Open problems}\label{sec:open}

We close with the questions left open by the two main mechanisms.

\subsection*{The split construction in every degree}
Does $S/I(d)$ have the Hilbert series $[(1-t^d)^6/(1-t)^4]_+$ for every
$d\ge1$? An affirmative answer would give an explicit sequence of six
forms in four variables with the predicted series in every degree, as
asked in~\cite[Section~2]{FLOS2018}. Theorem~\ref{thm:seven-intro}
leaves, for each fixed $\varepsilon>0$ and all sufficiently large $d$, only grades in the window
$(11/5-\varepsilon)d<q<(73/32+\varepsilon)d$, and
Proposition~\ref{prop:twenty} isolates what the passage from $d$ to
$d+1$ requires: in the grade below the first predicted zero, the map from the direct sum of the six
plane summands and the induced line map must both be injective, and in the next
grade the line map must be surjective onto the cokernel left by the
planes. The line map has a source of dimension $15b-20$, linear in $d$,
so the question is a rank problem of linear size once the plane
contributions are known.

\subsection*{Exactness of the inherited bound}
For $e\ge36$ and five general pairs of ternary linear forms, is
$\dim_k(S/I_e)_q=B_e(q)$ for every $q$? By
Proposition~\ref{prop:two-critical-degrees} this is equivalent to two
restriction ranks at the crossing $q\approx(7e+5)/4$. For
$e\equiv1\pmod4$ it is the invertibility of a single square matrix of
size $(15e+45)/4$. Theorems~\ref{thm:intro-uniform-hilbert}
and~\ref{thm:intro-terminal} show that equality holds outside the degrees
$\lfloor(7e+5)/4\rfloor-1\le q\le2e-3$ and that the error there is at
most $55$.

More generally, Shapiro's comparison~\cite[Question~4.1]{Shapiro2026}
takes the larger of the generic-form and pure-power values. For other
numbers of generators, variables, or partition types, does the maximum of
the pure-power value and the bound obtained by adding the transferred
relations~\eqref{eq:intro-transfer} to the generic-form count determine
the generic Hilbert function up to an error bounded independently of the
degree?

\subsection*{Perturbations of powers}
Question~\ref{q:rogers-szego} asks whether the full homogeneous
Rogers--Szegő polynomial completes five powers of linear forms to the
predicted series. Theorem~\ref{thm:intro-obstruction} rules out every
perturbation that retains a factor of multiplicity at least
$(15d+91)/16$, whereas Theorem~\ref{thm:seven-intro} shows that a
balanced split succeeds in a large range. Let $\mu(d)$ be the largest
positive integer $m$ such that some six forms of degree $d$ in four variables,
each divisible by the $m$th power of a linear form, have the predicted
series, with $\mu(d)=0$ if no such $m$ exists. Theorem~\ref{thm:intro-obstruction} gives
$\mu(d)\le5\lceil3d/16\rceil$, and Theorem~\ref{thm:seven-intro} gives
$\mu(d)\ge\lceil d/2\rceil$ for $d\le30$. Does $\mu(d)/d$ have a limit,
and is it larger than $1/2$?

\subsection*{The non-Lefschetz scheme}
Theorem~\ref{thm:non-lefschetz} determines the local structure of
$\mathcal N_{2e-2}$ on the union of the charts $U_i$. A complete
description would require the degenerate conic projections and the
rank drops of five-generator quotients, which the argument does not
classify. It would also be interesting to compare the schemes in
different degrees.

\subsection*{Equal degrees in four variables}
Appendix~\ref{sec:fixed-degrees} settles any number of general forms of
each degree $d\le9$ in four variables, and Theorem~\ref{thm:seven-intro}
settles six general forms of each degree $d\le30$. For $d\ge10$ and other
numbers of generators the equal-degree case in four variables remains
open in general. The endpoint method of Appendix~\ref{sec:endpoints}
applies whenever the last positive degree of the predicted series is
below $2d$, and then reduces a fixed degree to finitely many rank
certificates.

\appendix
\section{Exact certificates in four variables}\label{sec:fixed-degrees}

We record exact certificates for two fixed degrees in four variables.
The conjecture holds for $r\le n$, for $n=2$~\cite{Froberg1985}, for
$n=3$~\cite{Anick1986}, and for $r=n+1$, where it follows from the strong
Lefschetz property of monomial complete intersections, proved
independently by Stanley~\cite{Stanley1980} and
Watanabe~\cite{Watanabe1987}.  For forms of equal degree $d$, Hochster and
Laksov~\cite{HochsterLaksov1987} proved that the prediction is correct in
degree $d+1$; Aubry~\cite{Aubry1995}, Migliore and
Miró-Roig~\cite{MiglioreMiroRoig2003}, and Nenashev~\cite{Nenashev2017}
gave sufficient conditions in degree $d+d'$; and Boij, Dannetun, and
Lundqvist recently proved it through degree $d+2$ whenever $d>2$
\cite[Theorem~1]{BoijDannetunLundqvist2026}.  Pardue formulated an
equivalent conjecture on the reverse-lexicographic initial ideal of $n$
general forms~\cite{Pardue2010}, and Trung used it to prove Fröberg's
conjecture for $r\le n+2$ under conditions on the degrees~\cite{Trung2019}.  Blomenhofer and Casarotti give degree-by-degree independence criteria
for general forms of several prescribed degrees by studying joins
\cite[Theorem~4.1]{BlomenhoferCasarotti2025}. Small cases were checked by computer by Fröberg
and Hollman~\cite{FrobergHollman1994}.  For surveys and recent
developments see~\cite{FLOS2018,FrobergLundqvist2018,Froberg2025,
FrobergLofwall2026}.

Already for $n=4$ and equal degrees the conjecture is open in general.
Nicklasson proved it when $\binom{n+d-1}{d}-n<r\le\binom{n+d-1}{d}$
\cite[Corollary~1]{Nicklasson2017}.  Nicklasson also showed, by computer
verification that ideals generated by $m$th powers of general forms of
degree $d'$ behave like ideals of general forms of degree $md'$, that it
holds for $n=4$ and $d\in\{4,6,8,9\}$
\cite[Corollary~2]{Nicklasson2017}.  This method requires a factorization
$d=md'$ with $m,d'\ge2$, so it does not reach prime degrees.  For $d=5$, the
theorem of Boij, Dannetun, and Lundqvist gives the complete series once
$r\ge12$, since the predicted series then vanishes in degree $d+2=7$; for
$d=7$ it does so once $r\ge22$.  Together with the classical cases
$r\le5$, this leaves $6\le r\le11$ for quintics and $6\le r\le21$ for
septics.

\begin{theorem}\label{thm:main}
Let $k$ be a field of characteristic zero and let
$S=k[x_1,x_2,x_3,x_4]$.  If $d\in\{5,7\}$, then for every integer $r\geq1$
there is a nonempty Zariski-open subset
$U_{d,r}\subseteq(S_d)^r$ such that every
$(F_1,\ldots,F_r)\in U_{d,r}$ satisfies
\begin{equation}\label{eq:main-series}
 \HS_{S/(F_1,\ldots,F_r)}(t)
 =\trunc{\frac{(1-t^d)^r}{(1-t)^4}}.
\end{equation}
\end{theorem}

\begin{corollary}\label{cor:four-variables}
Let $k$ be a field of characteristic zero.  For every $1\le d\le9$ and every
$r\ge1$, ideals generated by $r$ general forms of degree $d$ in
$k[x_1,x_2,x_3,x_4]$ have the Hilbert series
$[(1-t^d)^r/(1-t)^4]_+$.
\end{corollary}

\begin{proof}
The case $d=1$ is trivial and $r\le5$ is classical.  For $d=2$ and $r\ge5$
the predicted series vanishes in degree $3=d+1$, so the result of Hochster
and Laksov applies.  For $d=3$ and $r\ge6$ it vanishes in degree $5=d+2$,
so the theorem of Boij, Dannetun, and Lundqvist applies.  The degrees
$4,6,8,9$ are Nicklasson's Corollary~2, stated over $\mathbb C$; since the
relevant rank conditions are minors of matrices defined over $\mathbb Q$,
they hold on a nonempty open set over every field of characteristic zero.
The degrees $5$ and $7$ are Theorem~\ref{thm:main}.
\end{proof}

The first equal degree in four variables not covered by these results is
$d=10$.

The proof of Theorem~\ref{thm:main} is computer-assisted but exact; no
floating-point arithmetic or random sampling is involved.  It reduces each
degree to finitely many \emph{endpoint} matrices
(Section~\ref{sec:endpoints}).  For each of them we give an integral
specialization together with a maximal square minor that is nonsingular
modulo~$2$, hence has odd determinant.  Certificates of this kind are
routine in principle: a maximal-rank computation modulo a prime at an
integral point proves the corresponding statement in characteristic zero,
and the obstacle is size.  The contribution here is a complete,
independently checkable certificate for both degrees.  The accompanying
programs use only the Python standard library.  In the course of the proof
the certificates already give, in each degree, a single explicit nested
list of forms whose prefixes of length $r\ge6$ all have the predicted
series (Section~\ref{sec:structured}).

\section{Generic ranks and endpoint propagation}\label{sec:endpoints}

For $j\geq0$ put
\[
 N_j=\dim_k S_j=\binom{j+3}{3},
\]
and set $S_j=0$ for $j<0$.  For an ordered $r$-tuple
$\boldsymbol{F}=(F_1,\ldots,F_r)\in(S_d)^r$, consider
\begin{equation}\label{eq:mu}
 \mu^{(d)}_{r,j}(\boldsymbol{F}):S_{j-d}^{\,r}\longrightarrow S_j,
 \qquad (h_1,\ldots,h_r)\longmapsto \sum_{i=1}^{r}h_iF_i.
\end{equation}
If $I=(F_1,\ldots,F_r)$, then
\begin{equation}\label{eq:hf-rank}
 \dim_k(S/I)_j=N_j-\rank\mu^{(d)}_{r,j}(\boldsymbol{F}).
\end{equation}

\begin{lemma}[Openness of a rank condition]\label{lem:open}
For fixed $d,r,j$, the locus in $(S_d)^r$ on which
$\rank\mu^{(d)}_{r,j}\geq s$ is Zariski open.  If an integral
specialization has rank at least $s$ after reduction modulo a prime $p$,
then this locus is nonempty over every field of characteristic zero.
\end{lemma}

\begin{proof}
In monomial bases, the matrix entries are linear polynomials in the
coefficients of the $F_i$.  The rank locus is the union of the principal
open sets defined by its $s\times s$ minors.  A minor that is nonzero modulo
$p$ at an integral specialization has a nonzero integer value, so its
determinant polynomial remains nonzero after base change to any
characteristic-zero field.
\end{proof}

\begin{lemma}[Endpoint propagation]\label{lem:endpoint}
Let $G_1,\ldots,G_b\in S_d$ and let $1\leq a\leq b$.  Suppose that
$\mu^{(d)}_{b,q}(G_1,\ldots,G_b)$ is injective and that
$\mu^{(d)}_{a,q+1}(G_1,\ldots,G_a)$ is surjective.  Then, for every
$a\leq r\leq b$, the prefix $G_1,\ldots,G_r$ has
\begin{enumerate}
\item $\mu^{(d)}_{r,j}$ injective for every $d\leq j\leq q$; and
\item $\mu^{(d)}_{r,j}$ surjective for every $j\geq q+1$.
\end{enumerate}
\end{lemma}

\begin{proof}
Deleting generator blocks from the source preserves column independence,
so $\mu^{(d)}_{r,q}$ is injective.  A nonzero syzygy in a lower degree,
multiplied componentwise by a nonzero monomial of complementary degree,
would give a nonzero syzygy in degree $q$ because $S$ is a domain.  This
proves (1).  Adding generator blocks can only enlarge the image, so
$\mu^{(d)}_{r,q+1}$ is surjective.  Hence $I_{q+1}=S_{q+1}$, and ideality
gives $I_j=S_j$ for every $j\geq q+1$.
\end{proof}

For $r\ge5$ let $q=q(d,r)$, the \emph{last positive degree}, be the largest
integer such that the coefficients of $t^0,\ldots,t^q$ in
$(1-t^d)^r/(1-t)^4$ are all positive; thus
$[(1-t^d)^r/(1-t)^4]_+$ is the part of this series of degree at most $q$.
For every endpoint block below, $q<2d$.  Hence, for $d\leq j\leq q$, the
coefficient of $t^j$ is
\begin{equation}\label{eq:low-coeff}
 N_j-rN_{j-d}.
\end{equation}
Within a block of generator counts on which $q$ is constant,
Lemma~\ref{lem:endpoint} therefore converts one injective endpoint at the
upper end and one surjective endpoint at the lower end into the full
truncated Hilbert series for every generator count in the block: the
quotient has dimension $N_j$ for $j<d$, dimension~\eqref{eq:low-coeff} for
$d\le j\le q$, and dimension $0$ for $j>q$.

\section{The equal-quintic slice}\label{sec:quintics}

For $d=5$, direct expansion of $(1-t^5)^r/(1-t)^4$ gives the endpoint
blocks in Table~\ref{tab:quintic-blocks}.  Every last positive degree
satisfies $q<10=2d$, so \eqref{eq:low-coeff} applies on the injective
range.  The single endpoint at degree $2d$ is the surjective check $(6,10)$.

\begin{table}[ht]
\centering
\caption{Endpoint blocks for equal quintics.}\label{tab:quintic-blocks}
\begin{tabular}{c@{\qquad}c@{\qquad}c@{\qquad}c}
\toprule
$r$-range & last positive $q$ & injective endpoint & surjective endpoint\\
\midrule
$6$        & 9 & $(6,9)$  & $(6,10)$\\
$7$--$8$   & 8 & $(8,8)$  & $(7,9)$\\
$9$--$11$  & 7 & $(11,7)$ & $(9,8)$\\
$12$--$20$ & 6 & $(20,6)$ & $(12,7)$\\
$21$--$55$ & 5 & independence in degree 5 & $(21,6)$\\
\bottomrule
\end{tabular}
\end{table}

We use the following ordered family in $\mathbb Z[x,y,z,w]$, writing
$x,y,z,w$ for $x_1,x_2,x_3,x_4$; every coefficient displayed is $+1$:
\begingroup
\small
\begin{align*}
g_1&=x^3zw+x^2y^2z+w^5,
&g_2&=x^2w^3+y^2z^2w+zw^4,\\
g_3&=x^4y+xz^3w+y^2z^3,
&g_4&=x^5+x^3w^2+xyzw^2,\\
g_5&=xyzw^2+y^3w^2+z^5,
&g_6&=x^2yw^2+y^5+yz^2w^2,\\
g_7&=x^2y^2z+x^2yzw+y^3zw,
&g_8&=x^4w+x^3y^2+x^3w^2,\\
g_9&=x^5+x^2z^3+y^2zw^2,
&g_{10}&=x^3w^2+xz^4+xz^2w^2,\\
g_{11}&=x^2y^3+y^3zw+y^2w^3,
&g_{12}&=x^3yz+xy^4+z^4w,\\
g_{13}&=x^3y^2+xy^3w+z^2w^3,
&g_{14}&=x^3yw+xy^2z^2+yz^4,\\
g_{15}&=x^5+x^3zw+xy^4,
&g_{16}&=x^3yw+xz^2w^2+y^4z,\\
g_{17}&=x^4w+x^2w^3+xyw^3,
&g_{18}&=x^2z^2w+y^3z^2+yz^3w,\\
g_{19}&=x^3y^2+x^2z^3+y^4z,
&g_{20}&=x^3yw+x^2y^3+yzw^3,\\
g_{21}&=x^3z^2+x^2y^2w+x^2z^2w.&&
\end{align*}
\endgroup

Rows and columns of every multiplication matrix are ordered by descending
lexicographic order on exponent tuples.  The column order is first by form
index and then by multiplier monomial.

\begin{proposition}[Quintic rank certificate]\label{prop:quintic-certificate}
The ten matrices in Table~\ref{tab:quintic-ranks} have the displayed rank
over $\FF_2$.  Each rank is maximal; the same ranks are obtained over
$\FF_{101}$.
\end{proposition}

\begin{table}[ht]
\centering
\caption{Exact endpoint ranks for equal quintics.}\label{tab:quintic-ranks}
\begin{tabular}{c@{\quad}c@{\quad}c@{\quad}c@{\qquad}l}
\toprule
$r$ & $j$ & matrix shape & $\rank_{\FF_2}$ & role\\
\midrule
 6 &  9 & $220\times210$ & 210 & injective\\
 6 & 10 & $286\times336$ & 286 & surjective\\
 7 &  9 & $220\times245$ & 220 & surjective\\
 8 &  8 & $165\times160$ & 160 & injective\\
 9 &  8 & $165\times180$ & 165 & surjective\\
11 &  7 & $120\times110$ & 110 & injective\\
12 &  7 & $120\times120$ & 120 & surjective\\
20 &  6 & $ 84\times 80$ &  80 & injective\\
21 &  6 & $ 84\times 84$ &  84 & surjective\\
21 &  5 & $ 56\times 21$ &  21 & independent forms\\
\bottomrule
\end{tabular}
\end{table}

\begin{proof}
The source file \texttt{froberg\_quintics\_verifier.py} constructs the
matrices from the displayed exponent vectors.  Over $\FF_2$ it records
pivot rows and columns for a maximal square minor and verifies that minor
again.  It then repeats the rank computation over $\FF_{101}$.  All
arithmetic is exact and deterministic.
\end{proof}

\begin{corollary}\label{cor:quintic-charzero}
Every matrix in Table~\ref{tab:quintic-ranks} has maximal rank over every
field of characteristic zero.
\end{corollary}

\begin{proof}
Each recorded maximal minor is nonzero modulo $2$, so its determinant is an
odd, nonzero integer.
\end{proof}

The boundary case $r=5$ is separate from the sparse family.

\begin{proposition}\label{prop:quintic-r5}
Five general quintics in four variables have Fröberg's predicted Hilbert
series.
\end{proposition}

\begin{proof}
The statement is the classical $r=n+1$ case supplied by the strong Lefschetz
property for monomial complete
intersections~\cite{Stanley1980,Watanabe1987}.
For completeness, we also record an independent exact certificate, which is
useful for the self-contained computational package.
Use the integral tuple
\[
x^5,\quad y^5,\quad z^5,\quad w^5,\quad (x+y+z+w)^5.
\]
The degree-$9$, $10$, and $11$ multiplication matrices have shapes
$220\times175$, $286\times280$, and $364\times420$.  Their ranks over
$\FF_{101}$ and $\FF_{1009}$ are $175$, $270$, and $364$ respectively, so
some $175\times175$, $270\times270$, and $364\times364$ minor of the
respective matrix is a nonzero integer.  Hence the degree-$9$ map is injective and the degree-$11$ map is
surjective over every characteristic-zero field.  The five generators are
linearly independent in characteristic zero, since the coefficient of
$x^4y$ in $(x+y+z+w)^5$ is $5$.  In the domain of the degree-$10$ map, the ten
pairwise Koszul syzygies---the
tuples in $S_5^{\,5}$ with $F_j$ in position $i$ and $-F_i$ in position
$j$---are therefore linearly independent: a linear relation among them
would restrict, on each summand, to a linear relation among the five forms.
Thus the rank is at most $280-10=270$, and the modular computation
supplies the matching lower bound.  Injectivity in degree $9$ propagates
downward by the syzygy-lifting argument in Lemma~\ref{lem:endpoint}, while
surjectivity in degree $11$ propagates upward by ideality.  Together with the
exact degree-$10$ rank and~\eqref{eq:hf-rank}, this gives the Hilbert function
\[
(1,4,10,20,35,51,64,70,65,45,16,0,\ldots),
\]
which is the positive truncation of $(1-t^5)^5/(1-t)^4$.
\end{proof}

\begin{proof}[Proof of Theorem~\ref{thm:main} for $d=5$]
For $r\leq4$, a general tuple is a regular sequence, and
Proposition~\ref{prop:quintic-r5} handles $r=5$.  For $6\leq r\leq20$,
combine Corollary~\ref{cor:quintic-charzero}, the first four rows of
Table~\ref{tab:quintic-blocks}, and Lemma~\ref{lem:endpoint}.

For $21\leq r\leq55$, the first twenty-one forms are linearly independent
in the $56$-dimensional space $S_5$ and their degree-$6$ multiplication map
is surjective.  Extend them by monomials to a basis $g_1,\ldots,g_{56}$ of
$S_5$.
Every prefix through $g_{55}$ is injective in degree $5$, while surjectivity
in degree $6$ persists after adding generators.  Lemma~\ref{lem:endpoint}
applies with $q=5$.  At $r=56$ the forms span $S_5$, and for $r>56$ the
conclusion persists after adding forms.

For $r\leq5$, the classical results above already supply a nonempty open
subset.  For each fixed $r\geq6$, only finitely many maximal-minor conditions
were used.  The explicit specialization makes their intersection nonempty, so
Lemma~\ref{lem:open} gives the required nonempty open subset $U_{5,r}$.
\end{proof}

\section{The equal-septic slice}\label{sec:septics}

The septic argument uses the same endpoint mechanism but a longer nested
family.  The public certificate
\texttt{froberg\_septics\_certificate.json} contains 120 integral forms of
degree $7$, encoded by their exponent supports, and the maximal-minor data
for all fifteen endpoints below.  Monomial bases are ordered as in
Section~\ref{sec:quintics}.

\begin{proposition}[Septic rank certificate]\label{prop:septic-certificate}
For one common ordered family $G_1,\ldots,G_{120}$ in
$\mathbb Z[x_1,x_2,x_3,x_4]_7$, the fifteen multiplication matrices in
Table~\ref{tab:septic-ranks} have maximal rank over $\FF_2$.  For every
matrix, the certificate records a nonsingular maximal square minor.
\end{proposition}

\begin{table}[ht]
\centering
\caption{Exact endpoint ranks for equal septics.}\label{tab:septic-ranks}
\small
\begin{tabular}{c@{\quad}c@{\quad}c@{\quad}c@{\qquad}l}
\toprule
$r$ & $j$ & matrix shape & $\rank_{\FF_2}$ & role\\
\midrule
  6 & 13 & $560\times504$ & 504 & injective\\
  6 & 14 & $680\times720$ & 680 & surjective\\
  8 & 12 & $455\times448$ & 448 & injective\\
  7 & 13 & $560\times588$ & 560 & surjective\\
 10 & 11 & $364\times350$ & 350 & injective\\
  9 & 12 & $455\times504$ & 455 & surjective\\
 14 & 10 & $286\times280$ & 280 & injective\\
 11 & 11 & $364\times385$ & 364 & surjective\\
 21 &  9 & $220\times210$ & 210 & injective\\
 15 & 10 & $286\times300$ & 286 & surjective\\
 41 &  8 & $165\times164$ & 164 & injective\\
 22 &  9 & $220\times220$ & 220 & surjective\\
119 &  7 & $120\times119$ & 119 & injective\\
 42 &  8 & $165\times168$ & 165 & surjective\\
120 &  7 & $120\times120$ & 120 & surjective\\
\bottomrule
\end{tabular}
\end{table}

\begin{proof}
The source file \texttt{froberg\_septics\_verifier.py} independently
enumerates monomials, reconstructs each matrix from the 120 recorded forms,
computes its rank over $\FF_2$, extracts the selected square submatrix, and
verifies that the submatrix has full rank.  It also
recomputes the boundary-coefficient table used below.  Thus every selected
minor has determinant $1$ modulo $2$ and hence an odd, nonzero integer
determinant.
\end{proof}

Direct integer expansion of $(1-t^7)^r/(1-t)^4$ gives
Table~\ref{tab:septic-blocks}.  In every finite block the last positive
degree is less than $14=2d$, so no two-generator Koszul term appears in
\eqref{eq:low-coeff}.  For $r=6$ the surjective endpoint lies at degree
$14=2d$; the untruncated coefficient there is negative, and surjectivity
gives vanishing.

\begin{table}[ht]
\centering
\caption{Endpoint blocks for equal septics.}\label{tab:septic-blocks}
\small
\begin{tabular}{c@{\quad}c@{\quad}c@{\quad}c}
\toprule
$r$-range & last positive $q$ & injective endpoint & surjective endpoint\\
\midrule
$6$          & 13 & $(6,13)$   & $(6,14)$\\
$7$--$8$     & 12 & $(8,12)$   & $(7,13)$\\
$9$--$10$    & 11 & $(10,11)$  & $(9,12)$\\
$11$--$14$   & 10 & $(14,10)$  & $(11,11)$\\
$15$--$21$   &  9 & $(21,9)$   & $(15,10)$\\
$22$--$41$   &  8 & $(41,8)$   & $(22,9)$\\
$42$--$119$  &  7 & $(119,7)$  & $(42,8)$\\
$r\geq120$   &  6 & ---        & $(120,7)$\\
\bottomrule
\end{tabular}
\end{table}

\begin{proof}[Proof of Theorem~\ref{thm:main} for $d=7$]
For $r\leq4$, a general tuple is a regular sequence.  The case $r=5=n+1$
is the classical strong Lefschetz case~\cite{Stanley1980,Watanabe1987}.

For each finite block $[a,b]$ in Table~\ref{tab:septic-blocks},
Proposition~\ref{prop:septic-certificate} makes
$\mu^{(7)}_{b,q}$ injective and $\mu^{(7)}_{a,q+1}$ surjective over every
characteristic-zero field.  Lemma~\ref{lem:endpoint} gives injectivity
through degree $q$ and surjectivity from degree $q+1$ for every prefix
length $a\leq r\leq b$.  Equation~\eqref{eq:hf-rank}, together with
\eqref{eq:low-coeff} in degrees $d\leq j\leq q$ and vanishing from degree
$q+1$, then gives the positive truncation.

At $r=120$, the last row of Table~\ref{tab:septic-ranks} says that the forms
span the $120$-dimensional space $S_7$.  The quotient therefore vanishes
from degree $7$ onward, and adding further generators preserves this fact.

For $r\leq5$, the classical results above already supply a nonempty open
subset.  For each fixed $r\geq6$, the conditions used are the nonvanishing of
finitely many minors.  The first $\min(r,120)$ forms of the common integral
family, followed by arbitrary forms when $r>120$, give a point at which all
of them are nonzero.  Lemma~\ref{lem:open} therefore
gives the nonempty Zariski-open set $U_{7,r}$.
\end{proof}

\begin{remark}
The septic certificate proves a second complete fixed-degree slice; it does
not establish a pattern for all odd degrees.  The two certificates alone make no assertion about other degrees,
mixed degrees, or positive characteristic; the additional degrees in
Corollary~\ref{cor:four-variables} follow from the cited results.
\end{remark}

\section{Nested families}\label{sec:structured}

The certificates give more than a separate statement for each generator
count: in each degree they consist of prefixes of one ordered list.

\begin{proposition}\label{prop:simultaneous}
Fix $d\in\{5,7\}$, let $N=\binom{d+3}{3}$, and let $k$ be a field of
characteristic zero.
\begin{enumerate}
\item For $d=5$ let $G_1,\ldots,G_{56}$ be the forms $g_1,\ldots,g_{21}$ of
Section~\ref{sec:quintics} followed by monomials completing them to a basis
of $S_5$; for $d=7$ let $G_1,\ldots,G_{120}$ be the certificate family of
Section~\ref{sec:septics}.  Then for every $6\le r\le N$,
\[
 \HS_{S/(G_1,\ldots,G_r)}(t)=\trunc{(1-t^d)^r/(1-t)^4}.
\]
\item There is an ordered list $G'_1,\ldots,G'_N\in\mathbb
Z[x_1,x_2,x_3,x_4]_d$ for which the same equality holds for every
$1\le r\le N$.
\end{enumerate}
\end{proposition}

\begin{proof}
(1) In the proofs of Theorem~\ref{thm:main} for $d=5$ and $d=7$, every
endpoint matrix is built from a prefix of the stated list, and the
propagation in Lemma~\ref{lem:endpoint} is applied to prefixes of that
same list.

(2) Let $X=(S_d)^N$ over $\mathbb Q$ and let $\pi_r:X\to(S_d)^r$ be the
projection to the first $r$ forms.  By Theorem~\ref{thm:main}, each
$\pi_r^{-1}(U_{d,r})$ is a nonempty open subset of the irreducible space
$X$, so their intersection over $1\le r\le N$ is nonempty and open.  It
contains a rational point, since a nonzero polynomial over the infinite
field $\mathbb Q$ does not vanish on all of $\mathbb Q^{\dim X}$.  Clearing
denominators form by form does not change any prefix ideal.  The ranks of
the multiplication matrices, which have rational entries, do not change
under extension from $\mathbb Q$ to $k$.
\end{proof}

The list in part~(2) comes from finitely many open conditions, and its
coefficients are not controlled.  Neither part provides a rule that works
uniformly in the degree.  Section~\ref{sec:obstruction} examines one
natural attempt at such a rule for six forms.

\section{An integral basis for the conic residual spaces}
\label{sec:terminal-basis}

We give the basis used in Section~\ref{sec:terminal}; this also proves
the dimension formula there without a separate vanishing theorem.
Keep $C=XZ-Y^2$ and $P_i=(1,i,i^2)$, $0\le i\le4$. Fix $j\ge0$ and put
$D=10j+10$, $r=4j+4$, and $Q(s)=\prod_{i=0}^4(s-i)$.

For $p(s)=\sum_{k=0}^{2d}a_ks^k$, define its homogeneous lift by
\[
 \psi_d(p)=\sum_{k=0}^{2d}
 a_kX^{d-\lceil k/2\rceil}Y^{k\bmod2}Z^{\lfloor k/2\rfloor}.
\]
Restriction to $(X,Y,Z)=(1,s,s^2)$ recovers $p$. Restricting to the
conic, subtracting this lift, and dividing by $C$ gives the unique
expansion of every degree-$D$ form
\begin{equation}\label{eq:conic-basis-expansion}
 F=\sum_{h=0}^{\lfloor D/2\rfloor}C^h\psi_{D-2h}(p_h),
 \qquad \deg p_h\le2D-4h.
\end{equation}
The construction is integral: polynomial division by $Y^2-XZ$ is
monic and leaves normal monomials with $Y$-exponent at most one.

On $X=1$, write $s=Y$ and $c=Z-Y^2$. Define
\[
 T_a(p)=\sum_{k\ge2a}\binom{\lfloor k/2\rfloor}{a}a_ks^{k-2a}.
\]
Then $\psi_d(p)(1,s,s^2+c)=\sum_a c^aT_a(p)(s)$. The coefficient of
$c^h$ in \eqref{eq:conic-basis-expansion} is
$f_h=p_h+\sum_{\ell<h}T_{h-\ell}(p_\ell)$.
Since $(s-i,c)$ are regular local coordinates at $P_i$, membership
in all five $\mathfrak p_i^r$ is equivalent to
\begin{equation}\label{eq:conic-basis-divisibility}
 Q(s)^{r-h}\mid f_h(s)\qquad(0\le h<r).
\end{equation}

Given $p_0,\ldots,p_{h-1}$, put
$u_h=\sum_{\ell<h}T_{h-\ell}(p_\ell)$. All solutions of
\eqref{eq:conic-basis-divisibility} within the degree bound are
\[
 p_h=\operatorname{rem}(-u_h,Q^{r-h})+Q^{r-h}a_h,
 \qquad\deg a_h\le h\quad(0\le h<r).
\]
Indeed, $2D-4h-5(r-h)=h$, so the remainder fits the degree bound
and the displayed quotient has precisely the stated freedom.
For $h\ge r$, choose $p_h$ freely with degree at most $2D-4h$.
These choices are triangular and unique. Setting one free coefficient
equal to one and all the others to zero therefore gives an integral
basis of $V_j$, of dimension
\[
 \sum_{h=0}^{r-1}(h+1)+\binom{D-2r+2}{2}
   =10j^2+25j+16.
\]
The code orders free coefficients by increasing $h$ and then increasing
power of $s$. In $V_0$, the last basis vector is $C^5$; omitting it
gives the fifteen-dimensional complement used in the lifting matrix.

\section{Reproducibility and characteristic scope}\label{sec:reproducibility}

The quintic and septic verification programs use only the Python~3 standard library.
After placing the three ancillary files in one working directory, running
\begin{center}
\texttt{python3 froberg\_quintics\_verifier.py}
\end{center}
and
\begin{center}
\texttt{python3 froberg\_septics\_verifier.py
-{}-certificate froberg\_septics\_certificate.json}
\end{center}
reconstructs the endpoint matrices and terminates with \texttt{PASS}.
The quintic program constructs its sparse family internally.  The septic
certificate is kept separate because it records 120 forms and the selected
minor indices for fifteen matrices.

For Lemma~\ref{lem:terminal-certificate}, use Python~3 with
\texttt{python-flint} version~0.9.0. The dependency list is
\texttt{requirements.txt} in \texttt{anc/terminal}. From that directory, run
\begin{center}
\texttt{python reproduce.py terminal-output}
\end{center}
with a new output directory and without Python's optimization flag.
Seven sequential stages reconstruct the bases, jet splittings, polynomial
matrix, ranks, sign certificates, polynomial kernel and finite boundary
matrices. Each stage stops on a failed assertion; the final report and
checksum manifest record the generated outputs. No historical data files
are required. The symbolic reduction in Section~\ref{sec:terminal},
including its degree bound and specialization argument, is an essential
part of the proof and is not inferred from the finite evaluations.
Section~\ref{sec:non-lefschetz} uses this terminal certificate and
otherwise follows from the displayed contraction identity, conic
interpolation, and local ideal calculations. Its transverse example
uses the fixed primary points of \eqref{eq:terminal-points}.

The arguments claim only characteristic zero.  A successful computation
modulo $2$ is used to show that a specified integer determinant is nonzero;
it is not a claim that the desired generic Hilbert series holds in
characteristic $2$.  Section~\ref{sec:obstruction} uses
characteristic zero in the apolarity pairing.

The directory \texttt{anc/seven} gives the matrix construction for
Theorem~\ref{thm:seven-intro}, the complete endpoint table, and
instructions for reproducing the ranks and the seven plane restrictions.

\section*{Acknowledgements}

Dongming Zhang contributed substantially to the development of this work but withdrew from co-authorship for personal reasons. The author thanks Ralf Fröberg for helpful correspondence,
 for drawing attention to Stanley's paper in connection with
 Proposition~\ref{prop:quintic-r5}, and for recommending his 2025 paper on ideals of generic
 forms. The author thanks Samuel Lundqvist for explaining the computational
 context of the fixed-degree cases and for suggesting the study of
 structured sequences of forms.
The author acknowledges GeWu SciLab for providing the collaborative
research environment in which this project was developed.

\section*{Data and code availability}

The fixed-degree ancillary material includes the two standard-library verification
programs and the exact septic certificate.  Together they reconstruct the finite rank calculations in
Sections~\ref{sec:quintics} and~\ref{sec:septics}, which also yield
Proposition~\ref{prop:simultaneous}.  The arguments in
Sections~\ref{sec:inherited}, \ref{sec:polyharmonic},
\ref{sec:uniform-hilbert}, and~\ref{sec:obstruction}
are algebraic and require no computation.
The separate directory \texttt{anc/terminal} contains all source files
and dependency information for the exact arithmetic calculation in
Section~\ref{sec:terminal}. Its code constructs the polynomial
certificates from the points and directions displayed in the paper.

\section*{Disclosure of automated assistance}

OpenAI's GPT-5.6 Sol, GPT-6 Astra, and GPT-6.2 Astra,
Anthropic's Claude Fable 5 and Claude Opus 5.5, and xAI's Grok 4.6
contributed to the development and preparation of this paper. Their
contributions included deriving proof strategies for the four-variable
construction and the ternary Hilbert function bounds, constructing
examples and exact rank certificates, checking proofs and related
literature, and drafting and revising the manuscript.

\end{document}